\documentclass[11pt,reqno,final]{amsart}
\usepackage{xcolor}
\usepackage[marginpar=2cm]{geometry}
\usepackage{setspace}
\usepackage{indentfirst}
\usepackage{tgtermes}
\usepackage[right,mathlines]{lineno}
\usepackage{mathrsfs,mathtools,amssymb,stmaryrd,tikz-cd}
\usepackage{enumitem}
\setlist[enumerate,1]{label=(\roman*),ref=(\roman*)}

\theoremstyle{plain}
\newtheorem{introthm}{Theorem}

\newtheorem{introQ}[introthm]{Question}

\newtheorem{theorem}{Theorem}[section]
\newtheorem{lemma}[theorem]{Lemma}
\newtheorem{proposition}[theorem]{Proposition}
\newtheorem{corollary}[theorem]{Corollary}

\theoremstyle{definition}
\newtheorem{definition}[theorem]{Definition}
\newtheorem{example}[theorem]{Example}

\theoremstyle{remark}
\newtheorem{remark}[theorem]{Remark}

\usepackage[backend=bibtex,bibencoding=utf8,sortcites,style=numeric,url=false,doi=true,eprint=true,isbn=false,hyperref=auto,backref=true]{biblatex}
\usepackage[notcite,notref,color,final]{showkeys}
\definecolor{refkey}{gray}{.50}
\definecolor{labelkey}{rgb}{1,0,0}

\usepackage[pagebackref=false,hidelinks,unicode=true,bookmarks=true,linktoc=page,pdfview=FitH,pdfpagelayout=OneColumn,final=true]{hyperref}

\allowdisplaybreaks[3]

\usepackage{microtype}
\usepackage{macros}

\title{Kapranov $L_{\infty}[1]$ algebras}

\author{Ruggero Bandiera}
\address{Università degli studi di Roma La Sapienza, Dipartimento di Matematica Guido Castelnuovo}
\email{bandiera@mat.uniroma1.it}

\author{Seokbong Seol}
\address{School of Mathematics, Korea Institute for Advanced Study}
\email{azuredream89@kias.re.kr}

\author{Mathieu Stiénon}
\address{Department of Mathematics, Pennsylvania State University}
\email{stienon@psu.edu}

\author{Ping Xu}
\address{Department of Mathematics, Pennsylvania State University}
\email{ping@math.psu.edu}

\thanks{The second author is supported by the KIAS Individual Grants MG090801 and MG090802 at Korea Institute for Advanced Study. The fourth author's research is partially supported by the National Science Foundation (award DMS-2302447) and the Simons Foundation (award MP-TSM-00002272).}

\begin{document}

\begin{abstract}
Given any Kähler manifold $X$, Kapranov discovered a nontrivial $L_\infty[1]$ algebra structure on $\Omega^{0,\bullet}_{X}(T^{1,0}X)$.
Motivated by this result, we introduce the notion of an $L_\infty[1]$ $\mathfrak{R}$-algebra, generalizing $L_\infty[1]$ algebras, where $\mathfrak{R}$ is a differential graded commutative algebra with unit.
We extend standard notions such as quasi-isomorphism and linearization to this setting and establish a linearization theorem.

As an application, we prove that any DG Lie algebroid $(\mathcal{L},Q_{\mathcal{L}})$ over a DG manifold $(\mathcal{M},Q)$ induces a distinguished isomorphism class of $L_\infty[1]$ $\mathfrak{R}$-algebra structures on $\Gamma(\mathcal{L})$, where $\mathfrak{R}$ is the DG commutative algebra $(C^\infty(\mathcal{M}),Q)$.
Their unary bracket is $Q_{\mathcal{L}}$, while their binary bracket is a cocycle representative of the Atiyah class of the DG Lie algebroid.
All these isomorphic $L_\infty[1]$ $\mathfrak{R}$-algebra structures are $\rR$-linearizable if and only if the Atiyah class of the DG Lie algebroid vanishes.
In contrast, the induced $L_\infty[1]$ $\mathbb{K}$-algebra structures on $\Gamma(\mathcal{L})$ are necessarily $\mathbb{K}$-linearizable, hence $\mathbb{K}$-homotopy abelian.
\end{abstract}

\maketitle

\tableofcontents


\section{Introduction}

In his seminal paper \cite{MR1671737} on Rozansky--Witten invariants, Kapranov discovered a natural $L_\infty[1]$ algebra structure (over $\CC$) on the Dolbeault complex $\Omega_{X}^{0,\bullet}(T^{1,0}_X)$ of any Kähler manifold $X$.
This $L_\infty[1]$ algebra has a special feature: all its multibrackets $\lambda_k$ are $\Omega_{X}^{0,\bullet}$-multilinear except for the unary bracket $\lambda_1$, which is the Dolbeault differential operator $\overline{\partial}$.
Instead of being $\Omega_{X}^{0,\bullet}$-linear, the unary bracket $\lambda_1:\Omega_{X}^{0,\bullet}(T^{1,0}_X)\to\Omega_{X}^{0,\bullet+1}(T^{1,0}_X)$ satisfies the Leibniz rule; that is, $\big(\Omega_{X}^{0,\bullet}(T^{1,0}_X),\lambda_1=\overline{\partial}\big)$ is a DG module over the DG commutative algebra $\big(\Omega_{X}^{0,\bullet},\overline{\partial}\big)$.
The binary bracket $\lambda_2$ is the composition of the wedge product $\OO^{0,k}(T^{1,0}_X)\otimes\OO^{0,l}(T^{1,0}_X)\xto{\wedge}\OO^{0,k+l}(T^{1,0}_X\otimes T^{1,0}_X)$ and the contraction $\OO^{0,k+l}(T^{1,0}_X\otimes T^{1,0}_X)\xto{q_2}\OO^{0,k+l+1}(T^{1,0}_X)$ with the Atiyah cocycle $q_2\in\OO^{0,1}\big(\Hom(T^{1,0}_X\otimes T^{1,0}_X,T^{1,0}_X)\big)$.
The Atiyah class of a Kähler (or more generally complex) manifold $X$ captures the obstruction to the existence of a holomorphic connection on its holomorphic tangent bundle $T_X$.
This class is generally non-trivial, and its non-triviality gives rise to the Rozansky--Witten invariants of 3-manifolds \cite{MR1481135,MR2771578,MR1671725,MR2661534,arXiv:math/0404360}.

Subsequently, one of the authors \cite{MR3579974,MR3622306} proved that the $L_\infty[1]$ algebras discovered by Kapranov are homotopy abelian, i.e., isomorphic as an $L_\infty[1]$ algebra to the abelian $L_\infty[1]$ algebra \[ \big(\Omega_{X}^{0,\bullet}(T^{1,0}_X);\lambda_1=\overline{\partial},
\lambda_2=0,\lambda_3=0,\lambda_4=0,\cdots\big) .\]
Since the Atiyah class is generally non-trivial, this may at first seem puzzling.

Kapranov's construction was subsequently extended to all \emph{complex} manifolds and generalized in two directions: Lie pairs and DG manifolds.
A Lie pair $(L,A)$ is a pair of Lie algebroids over the same base manifold $M$ such that $A$ is a Lie subalgebroid of $L$.
It was proved in~\cite{MR4271478} that every Lie pair $(L,A)$ gives rise to an induced $L_{\infty}[1]$ algebra structure on $\Omega_{A}^{\bullet}(L/A):=\sections{\Lambda^{\bullet} A^{\vee}\tensor L/A}$, whose unary bracket is the Chevalley--Eilenberg differential corresponding to the Bott $A$-connection on $L/A$ and whose binary bracket is a cocycle representative of the Atiyah class of the Lie pair $(L,A)$.
In particular, the Lie pair $(\tangentcc{X},\tangentzo{X})$ arising from a complex manifold $X$ gives rise to an $L_{\infty}[1]$ algebra structure on the Dolbeault complex $\Omega_{X}^{0,\bullet}(T^{1,0}_X)$, extending Kapranov's construction for Kähler manifolds.
A parallel construction for DG manifolds was obtained in~\cite{MR3319134}: given any DG manifold $(\cM,Q)$, the space of vector fields $\XX(\cM)$ admits an $L_{\infty}[1]$ algebra structure, whose unary bracket is the Lie derivative $\liederivative{Q}$ in the direction of the homological vector field $Q$ and whose binary bracket is a cocycle representative of the Atiyah class of the DG manifold $(\cM,Q)$.

The present paper is devoted to answering the following question:
\begin{introQ}\label{introQA}
Are the $L_{\infty}[1]$ algebras arising from Lie pairs and/or DG manifolds homotopy abelian as are the $L_{\infty}[1]$ algebras discovered by Kapranov?
If so, how can the binary bracket simultaneously be trivial up to homotopy and be a cocycle representative of the possibly nontrivial Atiyah class?
\end{introQ}

Our central idea is the notion of $L_{\infty}[1]$ $\rR$-algebras, where $\rR=(\grR,d_{\grR})$ is a DG commutative algebra with unit.
An $L_{\infty}[1]$ $\rR$-algebra is an $L_{\infty}[1]$ algebra $\big(\LL;d_\LL,q_2,q_3,\cdots\big)$ such that (1) the cochain complex $(\LL,d_{\LL})$ is a DG module over the DG algebra $\rR=(\grR,d_{\grR})$ and (2) the multibrackets $q_n:\nst{\LL}{\grR}{n}\to\LL$ for $n\geq 2$ are graded symmetric graded $\grR$-multilinear maps, each of degree $1$.
Ordinary $L_{\infty}[1]$ algebras correspond to the case where $\rR$ is the field $\KK$ seen as a DG algebra concentrated in degree zero with (necessarily) trivial differential.
The usual notions of morphism, quasi-isomorphism, linearization, and homotopy abelianness extend naturally to $L_{\infty}[1]$ $\rR$-algebras.

An $L_{\infty}[1]$ $\rR$-algebra structure on a graded $\grR$-module $\LL$ is equivalent to a DG $\rR$-coalgebra structure $(S_{\grR}(\LL),Q)$ satisfying $Q(1_{\grR})=0$.
The $n$-th multibracket $q_n$ is the composition \[S^n_{\grR}(\LL)\into S_{\grR}(\LL)\xto{Q}S_{\grR}(\LL)\onto S^1_{\grR}(\LL)=\LL. \]
The space $\coDer_{\grR}(S_{\grR}(\LL))$ of $\cR$-coderivations of the $\grR$-coalgebra $S_{\grR}(\LL)$ admits a structure of a DG $\rR$-module with differential \[ \liederivative{Q}=[Q,-]:\coDer_{\grR}^{\bullet}(S_{\grR}(\LL)) \to\coDer_{\grR}^{\bullet+1}(S_{\grR}(\LL)) .\]

Evaluation of a coderivation $\phi$ of $S_{\grR}(\LL)$ at $1_{\grR}$ gives an element of $\LL$: $\phi(1_{\grR})\in\LL\subset S_{\grR}(\LL)$ --- see for instance \cite[Lemma~12.7.3]{MR4485797}.
Moreover, evaluation at $1_{\grR}$ defines a morphism \[ \ev_{1_{\grR}}:\coDer_{\grR}(S_{\grR}(\LL))\to\LL \] of DG $\rR$-modules.
Its kernel is the space of coderivations $\phi$ with vanishing $0$-th Taylor coefficient $\phi_{0}=0$,
which we denote by $F^1\coDer_{\grR}(S_{\grR}(\LL))$.
Thus, we have the exact sequence of DG $\rR$-modules
\begin{equation}\label{splitseq0}
0\to(F^1\coDer_{\grR}(S_{\grR}(\LL)),\liederivative{Q})
\to(\coDer_{\grR}(S_{\grR}(\LL)),\liederivative{Q})
\xrightarrow{\ev_1}(\LL, d_{\LL})\to 0
.\end{equation}

Our first main result is the following linearization theorem:
\begin{introthm}[Theorem~\ref{thm:main1}]
\label{introthmA}
An $L_{\infty}[1]$ $\rR$-algebra $(\LL;d_{\LL},q_{2},q_{3},\ldots)$ is linearizable as an $L_{\infty}[1]$ $\rR$-algebra, i.e., is isomorphic to the $L_{\infty}[1]$ $\rR$-algebra $(\LL;d_{\LL},0,0,\ldots)$, if and only if the short exact sequence \eqref{splitseq0} splits in the category of DG $\rR$-modules.
\end{introthm}

To answer Question~\ref{introQA}, we study $L_{\infty}[1]$ $\rR$-algebras from the perspective of changes of base DG commutative algebra $\rR$.

A morphism of DG commutative algebras $\rS=(\grS,d_{\grS})\xto{f}(\grR,d_{\grR})=\rR$ induces an adjunction between the categories $\dgSmod$ of DG $\rS$-modules and $\dgRmod$ of DG $\rR$-modules: the right adjoint functor is the restriction of scalars $\dgRmod\to\dgSmod$ (turning a DG $\rR$-module $(\LL,d_{\LL})$ into the DG $\rS$-module $(\LL,d_{\LL})$ with the $f$-induced $\rS$-module structure), while the left adjoint functor is the extension of scalars $\dgSmod\to\dgRmod$ (turning a DG $\rS$-module $(\LL, d_{\LL})$ into the DG $\rR$-module $(\grR\otimes_{\grS}\LL,d_{\grR}\otimes\id+\id\otimes d_{\LL})$).
This pair of adjoint functors can be upgraded to an adjunction between the categories $\LSA$ of $L_\infty[1]$ $\rS$-algebras and $\LRA$ of $L_\infty[1]$ $\rR$--algebras.
The restriction of scalars functor $\LRA\to\LSA$ preserves linearizability.
A linearizable $L_{\infty}[1]$ $\rR$-algebra is thus automatically a linearizable $L_{\infty}[1]$ $\KK$-algebra.
However, the converse is not true. For instance, the Dolbeault complex $\Omega_{X}^{0,\bullet}(T^{1,0}_X)$ need not be linearizable as an $L_\infty[1]$ $\rR$-algebra over the DG commutative algebra $\rR=(\Omega_{X}^{0,\bullet},\bar{\partial})$, which may give rise to nontrivial Rozansky—Witten invariants of 3-manifolds.
Nevertheless, it is necessarily linearizable \emph{as an $L_\infty[1]$ $\mathbb{C}$-algebra}, i.e., it is necessarily homotopy abelian.

Our second goal is to study the linearization problem for $L_{\infty}[1]$ $\rR$-algebras arising from geometry: arising from Lie pairs and DG manifolds.
We study them within the unifying framework provided by DG Lie algebroids.
DG Lie algebroids are Lie algebroid objects in the category of DG manifolds.
Both DG manifolds and Lie pairs naturally give rise to DG Lie algebroids.
On the one hand, given a DG manifold $(\cM,Q)$, its tangent bundle $\tangent{\cM}$ (together with the Lie derivative $\liederivative{Q}$ along the homological vector field $Q$) is a DG Lie algebroid denoted by $(\tangent{\cM},\liederivative{Q})$.
On the other hand, given a Lie pair $(L,A)$, the pull-back $\pi^{!}L$ of the Lie algebroid $L\to M$ along the morphism of graded manifolds $A[1]\to M$ (together with the Lie derivative $\liederivative{s}$ along a section $s$ of degree $1$ of $\pi^{!}L\to A[1]$ satisfying $[s,s]=0$) gives rise to a DG Lie algebroid denoted by $(\pi^{!}L,\liederivative{s})$ --- see Section~\ref{sec:DGLAdLP} for a more precise description.

Our main result is the following.
\begin{introthm}[Theorem~\ref{prop:QSRN}, Theorem~\ref{thm:DGPBW}, Theorem~\ref{thm:LAHA}]
\label{introthmB}
Let $(\cL,\QL)$ be a DG Lie algebroid over a DG manifold $(\cM,Q)$.
The homological vector field $Q$ is a derivation of the algebra $\smooth{\cM}$ of smooth functions on $\cM$ turning it into a DG commutative algebra, which we denote by $\rR=(\smooth{\cM},Q)$.
Furthermore, the operator $\QL$ acting on the space $\sections{\cL}$ of sections of the vector bundle $\cL\to\cM$ turns it into a DG module over the DG algebra $\rR$, which we denote by $(\sections{\cL},\QL)$.
\begin{enumerate}
\item The DG $\rR$-module $(\sections{\cL},\QL)$ admits an $L_{\infty}[1]$ $\rR$-algebra structure --- which we call a Kapranov $L_{\infty}[1]$ $\rR$-algebra --- whose binary bracket is a cocycle representative of the Atiyah class $\atiyahclassQL\in H^{1}(\sections{\cL^{\vee}\tensor\End(\cL)},\QL)$
of the DG Lie algebroid.
\item The Kapranov $L_{\infty}[1]$ $\rR$-algebra structure on $\sections{\cL}$ is $\rR$-linearizable if and only if $\atiyahclassQL=0$.
\item The induced $L_{\infty}[1]$ $\KK$-algebra structure on $\sections{\cL}$ is necessarily homotopy abelian.
\end{enumerate}
\end{introthm}

Specializing Theorem~\ref{introthmB} to the tangent DG Lie algebroid $(T\cM,\liederivative{Q})$ of a DG manifold $(\cM,Q)$, we recover one of the main theorems in~\cite{MR4393962}: the space of vector fields $(\XX(\cM),\liederivative{Q})$ admits a Kapranov $L_{\infty}[1]$ $\rR$-algebra structure whose unary bracket is the operator $\liederivative{Q}$ and whose binary bracket is a cocycle representative of the Atiyah class of the DG manifold.
We also obtain the new result that this $L_{\infty}[1]$ $\rR$-algebra structure on $\XX(\cM)$ is $\rR$-linearizable if and only if the Atiyah class of the DG manifold vanishes.
The induced $L_{\infty}[1]$ $\KK$-algebra structure on $\XX(\cM)$ is necessarily homotopy abelian.

We next consider the DG Lie algebroids $(\pi^{!}L, \liederivative{s})$ arising from Lie pairs $(L,A)$
described above. We prove the following.
\begin{introthm}[Theorem~\ref{thm:LPHA}, Corollary~\ref{cor:LPHA}]
\label{introthmC}
Let $(\pi^{!}L,\liederivative{s})$ be a DG Lie algebroid arising from a Lie pair $(L,A)$ and some appropriate section $s$.
The DG commutative algebra of functions on its base manifold $A[1]$, which we denote by $\rR$, is the Chevalley-Eilenberg algebra $(\Omega_{A}^{\bullet},d_{\CE})$ associated with the Lie algebroid $A$.
\begin{enumerate}
\item A Kapranov $L_{\infty}[1]$ $\rR$-algebra $\sections{\pi^{!}L}$ arising from the DG Lie algebroid $(\pi^{!}L,\liederivative{s})$ via Theorem~\ref{introthmB} is quasi-isomorphic to a Kapranov $L_{\infty}[1]$ $\rR$-algebra $\Omega_A(L/A)$ arising from the Lie pair $(L,A)$.
\item The Kapranov $L_{\infty}[1]$ $\rR$-algebra $\Omega_{A}(L/A)$ is linearizable if and only if the Atiyah class $\alpha_{(L,A)}$ of the Lie pair vanishes.
\item The induced $L_{\infty}[1]$ $\KK$-algebra structure on $\Omega_{A}(L/A)$ is necessarily homotopy abelian.
\end{enumerate}
\end{introthm}

In the particular case of the Lie pair $(L,A):=(\tangentcc{X},\tangentzo{X})$ arising from a complex manifold $X$, the pull-back DG Lie algebroid $\pi^{!}L\to A[1]$ is exactly the tangent DG Lie algebroid
$\big(T(\tangentzo{X}[1]), \liederivative{\bar{\partial}}\big)$ of the DG manifold $(\tangentzo{X}[1],\bar{\partial})$.
The Atiyah class of this Lie pair is known to coincide with the Atiyah class of the holomorphic tangent bundle $T_X$ (i.e., the original one defined by Atiyah in~\cite{MR86359}) --- see~\cite{MR3439229}.
Consequently, we obtain the following.
\begin{introthm}[Corollary~\ref{cor:complexmanifold}]
\label{introthmD}
Let $X$ be a complex manifold, and let $\rR$ be the DG commutative algebra $(\Omega_{X}^{0,\bullet},\bar{\partial})$.
\begin{enumerate}
\item \label{item:introCPX1}
The space of vector fields $\XX(\tangentzo{X}[1])$ on the DG manifold $(\tangentzo{X}[1],\bar{\partial})$ admits a Kapranov $L_{\infty}[1]$ $\rR$-algebra structure whose unary bracket is the operator $\liederivative{\bar{\partial}}$ and whose binary bracket is a cocycle representative of the Atiyah class of the DG manifold $(\tangentzo{X}[1],\bar{\partial})$.
\item \label{item:introCPX2}
The Dolbeault complex $(\Omega_{X}^{0,\bullet}(\tangentoz{X}),\bar{\partial})$ admits a Kapranov $L_{\infty}[1]$ $\rR$-algebra structure whose binary bracket is a cocycle representative of the Atiyah class
(in Dolbeault cohomology) of the holomorphic tangent bundle $T_X$.
\item The Kapranov $L_{\infty}[1]$ $\rR$-algebras in~\ref{item:introCPX1} and~\ref{item:introCPX2} are quasi-isomorphic.
\item The Kapranov $L_{\infty}[1]$ $\rR$-algebra on $\Omega_{X}^{0,\bullet}(\tangentoz{X})$ is linearizable if and only if the Atiyah class of the holomorphic tangent bundle $T_X$ vanishes.
\item The induced $L_{\infty}[1]$ $\CC$-algebra structure on $\Omega_{X}^{0,\bullet}(\tangentoz{X})$ is necessarily homotopy abelian.
\end{enumerate}
\end{introthm}

A primary motivation for studying Kapranov $L_\infty[1]$ algebras is their relation to Rozansky--Witten theory of $3$-manifolds.
For ongoing work on applications to Rozansky--Witten theory of $L_\infty[1]$ $\rR$-algebras arising from DG Lie algebroids, see~\cite{Austin}.
Potential connections with the work of Qiu and Zabzine~\cite{arXiv:1110.5234} are another direction for future study.

\section{\texorpdfstring{$L_{\infty}[1]$}{L-infinity[1]} algebras over a DG commutative algebra base}

Throughout this section, we work over a field $\KK$ of characteristic zero.
We refer to~\cite{MR4485797} for general facts about $L_{\infty}$ algebras defined over the base field $\KK$.
Let $\rR=(\grR,d_{\grR})$ be a DGCA, which we assume to be unital and over $\KK$.
Our goal for this section is to study $L_\infty$ $\rR$-algebras.
These can be viewed approximately (in a sense that is made precise below) as $L_{\infty}$ algebra objects in the category of DG $\rR$-modules. 
In particular, we aim to establish a criterion for such an object to be linearizable, i.e., to be $L_\infty$ isomorphic to a DG $\rR$-module with a trivial $L_\infty$ algebra structure.
To simplify signs, we shall shift degrees and work in the isomorphic category of $L_{\infty}[1]$ algebras.

\subsection{Basic notions}
Recall that a DG $\rR$-module
is a graded $\grR$-module $\LL$ with a differential $d_{\LL}$
satisfying the Leibniz rule $d_{\LL}(rx)=d_{\grR}(r)x+(-1)^{i}rd_{\LL}(x)$
for all $x\in \LL$ and all $r\in \grR^i$,
where $\grR^{i}$ denotes the degree $i$ component of $\grR$.

\begin{definition}
An \emph{$L_{\infty}[1]$ $\rR$-algebra} is a DG $\rR$-module $(\LL, d_{\LL})$
together with a sequence of graded symmetric  $\grR$-multilinear maps of degree one
\[q_n:\nst{\LL}{\grR}{n}:=\overbrace{\LL\odot_{\grR} \cdots\odot_{\grR} \LL}^n
\to \LL,\qquad n\geq 2,\]
satisfying the identities
\begin{equation}\label{eq:Loo}
\sum_{i=1}^n\sum_{\sigma\in \shuffle{i}{n-i}}\sign(\sigma)\cdot
q_{n-i+1}(q_i(x_{\sigma(1)},\ldots,x_{\sigma(i)}),x_{\sigma(i+1)},\ldots,x_{\sigma(n)})=0
,\end{equation}
for all $n\ge2$ and all homogeneous $x_1,\ldots,x_n\in \LL$
--- here, the symbol $q_1$ denotes the coboundary operator $d_{\LL}$ and the symbol $\sign(\sigma)$ denotes the Koszul sign
of the permutation $\sigma$ of the vectors $x_1,\dots,x_n$, defined by the identity
$x_{\sigma(1)}\odot\cdots\odot x_{\sigma(n)}= \sign(\sigma) \,x_1\odot\cdots\odot x_n$
in $\nst{\LL}{\grR}{n}$.
The symbol $\shuffle{i}{n-i}$ denotes the set of $(i,n-i)$-unshuffles, i.e., the permutations $\sigma\in S_n$
of the ordered set of $n$ elements $\{1,2,\ldots,n\}$ satisfying $\sigma(1)<\cdots<\sigma(i)$
and $\sigma(i+1)<\cdots<\sigma(n)$.
\end{definition}

\begin{definition}
An \emph{$L_{\infty}[1]$ $\rR$-morphism}
\[F=(f_1,f_2,\ldots):(\LL, d_{\LL},q_2,q_3,\ldots)\to(\MM, d_{\MM},r_2,r_3,\ldots)\]
between $L_{\infty}[1]$ $\rR$-algebras $\LL$ and $\MM$ is a collection
of graded symmetric $\grR$-multilinear maps of degree zero
\[ f_n:\nst{\LL}{\grR}{n} \to \MM,\qquad n\ge1 ,\]
satisfying the identities (where $q_1:=d_{\LL}$, $r_1:=d_{\MM}$)
\begin{multline}\label{eq:LooMor}
\sum_{k=1}^n\sum_{i_1+\cdots+i_k=n}\frac{1}{k!}
\sum_{\sigma\in\Nshuffle{i_1}{i_{2}}{i_{k}}}\sign(\sigma)\cdot
r_k(f_{i_1}(x_{\sigma(1)},\ldots),\ldots,f_{i_k}(\ldots,x_{\sigma(n)}))=
\\= \sum_{j=1}^n\sum_{\sigma\in\shuffle{j}{n-j}}\sign(\sigma)\cdot
f_{n-j+1}(q_j(x_{\sigma(1)},\ldots,x_{\sigma(j)}),x_{\sigma(j+1)},\ldots,x_{\sigma(n)})
\end{multline}
for all $n\ge1$ and homogeneous $x_1,\ldots,x_n\in\LL$.
\end{definition}

The composition of $L_{\infty}[1]$ $\rR$-morphisms is defined by the same formula
as for ordinary $L_{\infty}[1]$ algebras; see, for example, \cite[Section~12.2 and Definition~11.5.2]{MR4485797}.
In particular, the standard argument \cite[Corollary~11.5.5]{MR4485797}
shows that a morphism of $L_{\infty}[1]$
$\rR$-algebras $F$ is an isomorphism if and only if so is its linear component $f_1$.

With these definitions, $L_{\infty}[1]$ $\rR$-algebras
with $L_{\infty}[1]$ $\rR$-morphisms form a category, which we denote by $\LRA$.

\begin{remark} \label{rem:UnderlyingK}
When $\rR=(\KK,0)$, this recovers the usual definitions; see for instance,~\cite{MR4485797}. Conversely, it is clear that an $L_\infty[1]$ $\rR$-algebra structure $(\LL,d_{\LL},q_2,q_{3},\ldots)$ on $\LL$ is equivalent to an $L_\infty[1]$ $\KK$-algebra structure such that moreover the linear part $d_{\LL}$ makes $\LL$ into a DG $\rR$-module and the higher brackets $q_n$, $n\ge2$, are $\grR$-multilinear. 
Similarly, an $L_\infty[1]$ $\rR$-morphism $F=(f_1,f_{2},\ldots):(\LL,d_{\LL},q_2,q_{3},\ldots)\to(\MM,d_{\MM},r_2,r_{3},\ldots)$ between $L_\infty[1]$ $\rR$-algebras is the same as an $L_\infty[1]$ $\KK$-morphism in the usual sense such that moreover the coefficients $f_n$ are $\grR$-multilinear for all $n\ge1$. This provides a faithful functor from the category $\LRA$ into the category $\LKA$. 
This functor, however, need not be full, and in particular it can happen that non-isomorphic $L_\infty[1]$ $\rR$-algebras become isomorphic when seen as $L_\infty[1]$ $\KK$-algebras: this point will be important in the sequel.
\end{remark}

\begin{remark} \label{rem:CoB}
More generally, it is well known that a morphism $\gamma:\mathfrak{S}=(\grS,d_{\grS})\to\mathfrak{R}=(\grR,d_{\grR})$ of DGCAs induces a pair of adjoint functors between the categories of DG $\mathfrak{S}$-modules and DG $\rR$-modules.
In one direction, we have the right adjoint functor $\dgRmod\to\dgSmod$, called the \emph{restriction of scalars functor}, sending a DG $\rR$-module $(\LL,d_{\LL})$ to the same $(\LL,d_{\LL})$, but seen as an $\grS$-module via $sx:=\gamma(s)x$ for all $s\in\grS$ and $x\in \LL$. 
In the other direction, we have the left adjoint functor $\dgSmod\to\dgRmod$, called the \emph{extension of scalars functor}, sending a DG $\rS$-module $(\LL,d_{\LL})$ to the DG $\rR$-module $(\grR\otimes_{\grS}\LL,d_{\grR}\otimes\id_{\LL}+\id_{\grR}\otimes d_{\LL})$.
These extend to a pair of adjoint functors between the categories $\LRA$ and $\LSA$.
In one direction, we have the restriction of scalars functor $\LRA\to\LSA$ sending the $L_\infty[1]$ $\rR$-algebra $(\LL,d_{\LL},q_2,q_{3},\ldots)$ to the same $(\LL,d_{\LL},q_2,q_{3},\ldots)$ seen as an $L_\infty[1]$ $\rS$-algebra (notice that $\grR$-multilinearity of the higher brackets implies $\grS$-multilinearity), and similarly for morphisms.
In the other direction, we have the extension of scalars functor $\LSA\to\LRA$, sending the $L_\infty[1]$ $\rS$-algebra $(\LL,d_{\LL},q_2,q_{3},\ldots)$ to the $L_\infty[1]$ $\rR$-algebra $(\grR\otimes_{\grS}\LL, d_{\grR}\otimes\id_{\LL}+\id_{\grR}\otimes d_{\LL},q^{\grR}_{2},q^{\grR}_{3},\ldots)$, with the higher brackets defined by
\begin{equation*}
q^{\grR}_{n}(r_1\otimes x_1,\ldots,r_n\otimes x_n)
:=(-1)^{\sum_i|r_i| + \sum_{i<j}|x_i||r_j|} r_1\cdots r_n\otimes q_n(x_1,\ldots,x_n)
\end{equation*}
for all $n \ge 2$, and similarly sending an $L_\infty[1]$ $\rS$-morphism \[F=(f_1,f_{2},\ldots):(\LL,d_{\LL},q_{2},q_{3},\ldots)\to(\MM,d_{\MM},r_{2},r_{3},\ldots)\] to the $L_\infty[1]$ $\rR$-morphism
\[F^\grR:(\grR\otimes_\grS \LL, d_{\grR}\otimes \id_{\LL}+\id_{\grR}\otimes d_{\LL},q^{\grR}_{2},q^{\grR}_{3},\ldots)\to(\grR\otimes_\grS \MM,d_{\grR}\otimes \id_{\MM}+\id_{\grR}\otimes d_{\MM}, r^{\grR}_{2},r^{\grR}_{3},\ldots)\]
with coefficients $F^\grR=(f^{\grR}_{1}, f^{\grR}_{2},\ldots)$ defined by 
\begin{equation*}
f^{\grR}_{n}(r_1\otimes x_1,\ldots,r_n\otimes x_n)
:=(-1)^{\sum_{i<j}|x_i||r_j|} r_1\cdots r_n\otimes f_n(x_1,\ldots,x_n) \, .
\end{equation*}
It is easy to check that indeed this defines a pair of adjoint functors between the categories $\LRA$ and $\LSA$.
Moreover, as in the previous remark, we point out that the restriction of scalars functor defines a faithful functor of the category $\LRA$ into the category $\LSA$, but this functor need not be full, and in particular it can happen that non-isomorphic $L_\infty[1]$ $\rR$-algebras become isomorphic when seen as $L_\infty[1]$ $\rS$-algebras.
\end{remark}

\begin{example}\label{ex:KapDG}
Kapranov DG manifolds, as studied in~\cite{MR4271478},
are a special case of $L_\infty[1]$ $\rR$-algebras.
Recall that a Kapranov DG manifold consists of a pair of smooth vector bundles $A$ and $E$
over a common base manifold together with a pair of homological vector fields on the
$\ZZ$-graded manifolds $A[1]\oplus E$ and $A[1]$ such that both the inclusion
$A[1]\into A[1]\oplus E$ and the projection $A[1]\oplus E\onto A[1]$ are morphisms of DG manifolds.
In this case, $A$ is necessarily a Lie algebroid, $E$ is necessarily an $A$-module, and
$\Omega_A^\bullet (E): = \sections{\Lambda^\bullet A^\vee\otimes E}$
is an $L_\infty[1]$ $\rR$-algebra, where $\rR$ is the DGCA $(\Omega_{A}^{\bullet}, d_{\CE})$.
Moreover, morphisms of Kapranov DG manifolds correspond to morphisms of $L_\infty[1]$ $\rR$-algebras.
Hence, for a fixed Lie algebroid $A$, the category of Kapranov DG manifolds of the form $A[1]\oplus E$ is a subcategory of $\LRA$.
\end{example}

\subsection{Alternative definitions}

As is well known, $L_\infty[1]$ algebras and their morphisms can be equivalently described in terms of DG coalgebras.
We outline this description over a general DGCA base.
Given a graded $\grR$-module $\LL$, we write
\begin{equation}\label{S_R}
S_{\grR}(\LL) := \grR\oplus \Big(\bigoplus_{n=1}^{\infty} \nst{\LL}{\grR}{n} \Big)
=\grR\oplus \overline{S}_{\grR}(\LL) \, .
\end{equation}
It is a graded cocommutative $\grR$-coalgebra (that is, a cocommutative coalgebra object in the symmetric monoidal category of graded $\grR$-modules) via the comultiplication
$\Delta: S_{\grR}(\LL)\to S_{\grR}(\LL)\otimes_{\grR} S_{\grR}(\LL)$
defined by $\Delta(1_{\grR}) = 1_{\grR}\otimes 1_{\grR}$ and
\begin{equation}\label{eq:Delta}
\Delta(x_1\odot\cdots\odot x_n)
=\sum_{i=0}^{n}\sum_{\sigma\in\shuffle{i}{n-i}}\sign(\sigma)\cdot (x_{\sigma(1)}
\odot\cdots\odot x_{\sigma(i)})\otimes(x_{\sigma(i+1)}\odot\cdots\odot x_{\sigma(n)}) \, .
\end{equation}
It has a natural counit $\eta:S_{\grR}(\LL)\to \grR$ and coaugmentation
$\epsilon:\grR\to S_{\grR}(\LL)$ coming from the splitting \eqref{S_R}. Its reduced $\grR$-coalgebra is $(\overline{S}_\grR(\LL),\overline{\Delta})$, with the reduced coproduct $\overline{\Delta}$ defined by $\overline{\Delta}(X)=\Delta(X)-X\otimes1_\grR-1_\grR\otimes X$, $\forall \ X\in\overline{S}_\grR(\LL)$. Finally, it is locally conilpotent, in the sense that $\overline{S}_\grR(\LL)=\bigcup_{p=1}^\infty\ker(\overline{\Delta}^p)$, where we denote by
\[\overline{\Delta}^{p}:\overline{S}_\grR(\LL)\to\overbrace{\overline{S}_{\grR}(\LL)\otimes_\grR\cdots\otimes_{\grR} \overline{S}_\grR(\LL)}^{p+1},\qquad p\ge1,\] 
the iterated powers of $\overline{\Delta}$. In fact, $(S_\grR(\LL),\Delta,\eta,\epsilon)$ is the cofree object cogenerated by $\LL$ in the category of graded cocommutative, counital, coaugmented and locally conilpotent $\grR$-coalgebras.

An $\grR$-coderivation $\coderQ:S_{\grR}(\LL)\to S_{\grR}(\LL)$ is an $\grR$-linear map
$\coderQ$ making the following diagram
\begin{equation}\label{diag:R-coDer}
\begin{tikzcd}
S_{\grR}(\LL) \arrow{rr}{\coderQ} \arrow{d}{\Delta} && S_{\grR}(\LL) \arrow{d}{\Delta} \\
S_{\grR}(\LL) \tensor_{\grR}S_{\grR}(\LL) \arrow{rr}{\coderQ\tensor \id + \id\tensor \coderQ}
&& S_{\grR}(\LL) \tensor_{\grR}S_{\grR}(\LL)
\end{tikzcd}
\end{equation}
commutative.
By cofreeness of $S_\grR(\LL)$, such a $\coderQ$ is completely determined by its corestriction
\[ p\coderQ=: q = (q_0,q_1,q_2,\ldots)\in\Hom^{\bullet}_{\grR}(S_{\grR}(\LL), \LL)
\cong\prod_{n\ge0}\Hom^{\bullet}_{\grR}(\nst{\LL}{\grR}{n}, \LL) ,\]
where $p:S_{\grR}(\LL)\to \LL$ is the projection coming from the splitting \eqref{S_R}
(and where $\nst{\LL}{\grR}{0}:= \grR$). Explicitly, we can reconstruct $\coderQ$
from $q$ as follows: $\coderQ(1_{\grR}) = q_0(1_{\grR})\in \LL\subset S_{\grR}(\LL)$
and, for $n\geq 1$ and $x_1,\ldots,x_n\in \LL$,
\begin{multline}\label{eq:coder}
\coderQ(x_1\odot\cdots \odot x_n) = q_0(1_{\grR})\odot x_1\odot\cdots\odot x_n +
\\+ \sum_{i=1}^n\sum_{\sigma\in \shuffle{i}{n-i}}\sign(\sigma)\cdot
q_i(x_{\sigma(1)},\ldots,x_{\sigma(i)})\odot x_{\sigma(i+1)}\odot\cdots\odot x_{\sigma(n)}.
\end{multline}
We denote by $\coDer_{\grR}(S_{\grR}(\LL))$ the graded Lie algebra of $\grR$-coderivations
of $S_{\grR}(\LL)$ with the graded commutator bracket, and by
\begin{equation*}
\cF^{p}\coDer_{\grR}(S_{\grR}(\LL))\subset \coDer_{\grR}(S_{\grR}(\LL))
\end{equation*}
the graded Lie subalgebra of coderivations $\coderQ$ satisfying $q_0=q_1=\cdots=q_{p-1}=0$.
Thus we have the following filtration of graded Lie algebras:
\[ \coDer_{\grR}(S_{\grR}(\LL))= \cF^{0}\coDer_{\grR}(S_{\grR}(\LL))
\supset \cF^{1}\coDer_{\grR}(S_{\grR}(\LL)) \supset \cdots\supset \cF^{p}\coDer_{\grR}(S_{\grR}(\LL))
\supset \cdots .\]

Given graded $\grR$-modules $\LL$ and $\MM$, a morphism of coaugmented counital graded $\grR$-coalgebras $F:S_{\grR}(\LL)\to S_{\grR}(\MM)$ is completely determined (by cofreeness of $S_\grR(\MM)$) by its corestriction
\[ pF=(f_1,f_2,\ldots)\in\Hom^0_{\grR}(\overline{S}_{\grR}(\LL),\MM)
\cong\prod_{n\ge1}\Hom^0_{\grR}(\nst{\LL}{\grR}{n},\MM) ,\]
and conversely, a family of graded symmetric degree zero $\grR$-multilinear maps
$f_n:\nst{\LL}{\grR}{n} \to \MM$, $n\ge1$, determines a morphism $F:S_{\grR}(\LL)\to S_{\grR}(\MM)$
as follows: $F\circ\epsilon_{\LL}=\epsilon_{\MM}$
(where $\epsilon_{\LL}:\grR\to S_{\grR}(\LL)$ and $\epsilon_{\MM}:\grR\to S_{\grR}(\MM)$
are the coaugmentations) and for $n\geq 1$, $x_1,\ldots,x_n\in \LL$,
\begin{equation}\label{eq:morph} F(x_1\odot\cdots\odot x_n) = \sum_{k=1}^n\frac{1}{k!}\sum_{i_1+\cdots+i_k=n}
\sum_{\sigma\in \Nshuffle{i_1}{i_{2}}{i_k}}\,\sign(\sigma)\,
f_{i_1}(x_{\sigma(1)},\ldots)\odot\cdots\odot f_{i_k}(\ldots,x_{\sigma(n)}) \, .
\end{equation}

If $(\LL, d_{\LL})$ is a DG $\rR$-module, there is an induced map
$\widetilde{d_{\LL}}:S_{\grR}(\LL)\to S_{\grR}(\LL)$ defined by
$\widetilde{d_{\LL}}(r)=d_{\grR}(r)$ if $r\in \grR \subset S_{\grR}(\LL)$
and for $n\geq 1$ and $x_1,\ldots,x_n\in \LL$,
\begin{equation*}
\widetilde{d_{\LL}}(x_1\odot\cdots\odot x_n) = \sum_{\sigma\in \shuffle{1}{n-1}}
\sign(\sigma)\cdot d_{\LL}(x_{\sigma(1)})\odot x_{\sigma(2)}\odot\cdots\odot x_{\sigma(n)}.
\end{equation*}

The well-definedness of $\widetilde{d_{\LL}}$ is not immediately obvious given that the symmetric powers of $\LL$ are taken over $\grR$ and $d_{\LL}$ is not $\grR$-linear; however, well-definedness follows from a simple computation using the Leibniz rule. 
It is also easy to check that $\widetilde{d_{\LL}}$ squares to zero and satisfies the coderivation identity expressed by diagram~\eqref{diag:R-coDer}. 

However, $\widetilde{d_{\LL}}$ is not an $\grR$-coderivation, as it is not $\grR$-linear; thus, it is not an element of $\coDer_{\grR}(S_{\grR}(\LL))$. 
Instead, it satisfies the Leibniz rule, making $(S_\grR(\LL),\widetilde{d_{\LL}})$ a DG $\rR$-module. 
Another simple computation, again using the Leibniz rule, shows that
if $\coderQ\in\coDer_{\grR}(S_{\grR}(\LL))$ is an $\grR$-coderivation,
then the graded commutator $[\widetilde{d_{\LL}},\coderQ]:S_{\grR}(\LL)\to S_{\grR}(\LL)$
is also an $\grR$-coderivation, and thus
$\liederivative{d_{\LL}}(-):=[\widetilde{d_{\LL}},-]$ defines a degree one operator
\[ \liederivative{d_{\LL}}:\coDer^{\bullet}_{\grR}(S_{\grR}(\LL))
\to\coDer^{\bullet+1}_{\grR}(S_{\grR}(\LL)) \]
which squares to zero by the graded Jacobi identity
$\liederivative{d_{\LL}}^2(-) = \frac{1}{2}[[\widetilde{d_{\LL}},\widetilde{d_{\LL}}],-]
= [(\widetilde{d_{\LL}})^2,-]=0$. This makes $(\coDer_\grR(S_\grR(\LL)),\liederivative{d_{\LL}},[-,-])$ a DG Lie $\rR$-algebra; that is, $(\coDer_\grR(S_\grR(\LL)),\liederivative{d_{\LL}})$ is a DG $\rR$-module and the commutator bracket $[-,-]$ is $\grR$-bilinear. Moreover, the graded Lie subalgebra $\cF^{p}\coDer_{\grR}(S_{\grR}(\LL))\subset \coDer_{\grR}(S_{\grR}(\LL))$
is a DG Lie $\rR$-subalgebra for all $p\geq 1$.

To go back to $L_\infty[1]$ $\rR$-algebras, we have seen that a degree one coderivation $\coderQ\in \cF^{2}\coDer^1_\grR(S_\grR(\LL))$ is completely determined by the family $p\coderQ=(q_2,q_{3},\ldots)$ of degree one maps $q_n:S^n_\grR(\LL)\to \LL$. Conversely, such a family determines $\coderQ$ according to Equation~\eqref{eq:coder}. Set $\delta:=\widetilde{d_{\LL}}+\coderQ$. Then 
\[\delta^2=\frac{1}{2}[\delta,\delta] =\liederivative{d_{\LL}}(\coderQ)+\frac{1}{2}[\coderQ,\coderQ]\, .\]
Thus, $\delta$ squares to zero precisely when $\coderQ$ is a solution of the Maurer-Cartan Equation in the DG Lie $\rR$-algebra $(F^2\coDer_\grR(S_\grR(\LL)),\liederivative{d_{\LL}},[-,-])$. 
Moreover, set $\Xi:= \liederivative{d_{\LL}}(\coderQ)+\frac{1}{2}[\coderQ,\coderQ]$.
By cofreeness of $S_\grR(\LL)$, we have $\Xi=0$ if and only if $p\Xi=0$. Direct computation using \eqref{eq:coder} shows that the latter identity precisely translates into the $L_\infty[1]$ identities \eqref{eq:Loo} for the maps $q_n$.
Summing up, this proves the following lemma.
\begin{lemma} \label{lem:MC}
Let $(\LL, d_{\LL})$ be a DG $\rR$-module. The following are in bijective correspondence:
\begin{enumerate}
\item $L_{\infty}[1]$ $\rR$-algebra structures on $(\LL, d_{\LL})$;
\item Maurer-Cartan elements of the DG Lie algebra
$(\cF^{2}\coDer_{\grR}(S_{\grR}(\LL)), \liederivative{d_{\LL}},\lie{\argument}{\argument})$;
\item DG $\rR$-coalgebra structures $(S_\grR(\LL),\delta,\Delta)$ on $(S_\grR(\LL),\Delta)$ such that $\delta(1_\grR)=0$ and $\delta(x)=d_{\LL}(x)$ for all $x\in \LL\subset S_\grR(\LL)$.
\end{enumerate}
\end{lemma}

We have also seen that a morphism of coaugmented $\grR$-coalgebras $F:S_\grR(\LL)\to S_\grR(\MM)$ is determined by a family $pF=(f_{1},f_{2},\ldots)$ of degree zero maps $f_n:S^n_\grR(\LL)\to \MM$, according to Equation~\eqref{eq:morph}.
If $(\LL,d_{\LL},q_{2},q_{3},\ldots)$ and $(\MM,d_{\MM},r_{2},r_{3},\ldots)$ are $L_\infty[1]$ $\rR$-algebras, then, by cofreeness of $S_\grR(\MM)$, we see that $F$ is a morphism between the corresponding DG $\rR$-coalgebras if and only if $pF(\widetilde{d_{\LL}}+\coderQ)=p(\widetilde{d_{\MM}}+\coderR)F$.
A direct computation using Equations~\eqref{eq:coder} and~\eqref{eq:morph} shows that the latter precisely translates into the relations \eqref{eq:LooMor} saying that the maps $f_n$ are the coefficients of an $L_\infty[1]$ $\rR$-morphism.
We then see that the correspondence $(\LL,d_{\LL},q_{2},q_{3},\ldots)\to(S_\grR(\LL),\widetilde{d_{\LL}}+\coderQ)$ extends to a fully faithful embedding from the category of $L_\infty[1]$ $\rR$-algebras into the category of DG cocommutative (counital, coaugmented, locally conilpotent) $\rR$-coalgebras.

The well-known homotopy transfer theorem also holds in the context of $L_{\infty}[1]$ algebras defined over an arbitrary \DGCA{} $\rR=(\grR,d_{\grR})$.
Recall that a contraction $(f,g,h)$ of the complex $(\LL, d_{\LL})$ onto the complex $(\MM, d_{\MM})$ consists of cochain maps $f:(\MM, d_{\MM})\to(\LL, d_{\LL})$ and $g:(\LL, d_{\LL})\to(\MM, d_{\MM})$ such that $gf=\id_{\MM}$, together with a cochain homotopy $h$ between $fg$ and $\id_{\LL}$, that is, a degree $-1$ map $h:\LL^\bullet \to \LL^{\bullet-1}$ such that $d_{\LL}h+hd_{\LL} = fg-\id_{\LL}$, with the additional side conditions $h^2=gh=hf=0$.
Such a contraction is often represented by the following diagram:
\[
\begin{tikzcd}
 h &   \arrow[loop, distance=30, in=167, out=193]   (\LL, d_{\LL})  \arrow[shift left]{r}{g} & \arrow[shift left]{l}{f}
(\MM, d_{\MM})  \, .
\end{tikzcd}
\]
If $(\LL, d_{\LL})$ and $(\MM, d_{\MM})$ are DG $\rR$-modules and $f,g,h$ are $\grR$-linear, we say that $(f,g,h)$ is an $\grR$-linear contraction of the DG $\rR$-module $(\LL, d_{\LL})$ onto the DG $\rR$-module $(\MM, d_{\MM})$ ($\grR$-linearity of $h$ is intended in the graded sense: $rh(x)=(-1)^{|r|}h(rx)$ for all $r\in \grR$ and $x\in \LL$).
In this case, there is an induced contraction $(S_{\grR}(f),S_{\grR}(g),H)$ of the DG $\rR$-module $(S_{\grR}(\LL),\widetilde{d_{\LL}})$ onto the DG $\rR$-module $(S_{\grR}(\MM),\widetilde{d_{\MM}})$, where the contracting homotopy $H$ is defined via the symmetrized tensor trick: $H(r)=0$ if $r\in \grR\subset S_{\grR}(\LL)$ and for $n\geq 1$, $x_1,\ldots,x_n\in \LL$,
\begin{multline*}
H(x_1\odot\cdots\odot x_n)= \frac{1}{n!}\sum_{i=1}^n\sum_{\sigma\in S_n} \sign(\sigma)\\ (-1)^{|x_{\sigma(1)}|+\cdots+|x_{\sigma(i-1)}|}fg(x_{\sigma(1)})\odot\cdots\odot fg(x_{\sigma(i-1)})\odot h(x_{\sigma(i)})\odot x_{\sigma(i+1)}\odot\cdots\odot x_{\sigma(n)}.  
\end{multline*}
In the following theorem, given $\coderQ\in\coDer_{\grR}(S_{\grR}(\LL))$, we denote by $\coderQ^{k}_{n}$ the composition 
\[
\begin{tikzcd}
\nst{\LL}{\grR}{n} \arrow[hook]{r} & S_{\grR}(\LL) \arrow{r}{\coderQ} & S_{\grR}(\LL) \arrow[two heads]{r} & \nst{\LL}{\grR}{k},
\end{tikzcd} 
\]
where the first and last maps are, respectively, the natural inclusion and projection arising from \eqref{S_R}. 

\begin{theorem}[Homotopy Transfer Theorem]
\label{thm:transfer}
Given an $L_{\infty}[1]$ $\rR$-algebra $(\LL, d_{\LL},q_2,q_3,\ldots)$ and an $\grR$-linear contraction $(f,g,h)$ of the DG $\rR$-module $(\LL, d_{\LL})$ onto the DG $\rR$-module $(\MM, d_{\MM})$, there is an induced $L_{\infty}[1]$ $\rR$-algebra structure $(\MM, d_{\MM},r_2,r_3,\ldots)$ on $(\MM, d_{\MM})$, together with $L_{\infty}[1]$ $\rR$-quasi-isomorphisms $F:\MM \to \LL$ and $G:\LL\to \MM$ extending $f$ and $g$, respectively. 

Explicitly, $f_1=f$, and for $n\ge2$, $x_1,\ldots,x_n\in \MM$, $f_n$ and $r_n$ are recursively defined by
\[ f_n(x_1,\ldots,x_n) = \sum_{k=2}^n\sum_{i_1+\cdots+i_k=n}\sum_{\sigma\in \Nshuffle{i_1}{i_{2}}{i_k}} \frac{\sign(\sigma)}{k!} hq_k(f_{i_1}(x_{\sigma(1)},\ldots),\ldots,f_{i_k}(\ldots,x_{\sigma(n)})),  \]
\[ r_n(x_1,\ldots,x_n) = \sum_{k=2}^n\sum_{i_1+\cdots+i_k=n}\sum_{\sigma\in \Nshuffle{i_1}{i_{2}}{i_k}} \frac{\sign(\sigma)}{k!} g_1q_k(f_{i_1}(x_{\sigma(1)},\ldots),\ldots,f_{i_k}(\ldots,x_{\sigma(n)})).  \]
The coefficients $g_n:\nst{\LL}{\grR}{n} \to \MM$ of $G$ are defined recursively by $g_1=g$ and, for $n\geq 2$, by \[ g_n = \sum_{k=1}^{n-1} g_k \coderQ^k_nH.  \] 
\end{theorem}

Note that $\operatorname{Leibniz}_{\infty}[1]$ $\rR$-algebras --- analogues of $L_\infty[1]$ $\rR$-algebras --- were introduced by Chen--Liu--Xiang \cite[Definition~3.5]{MR4031996}.

\subsection{Linearizability}

We finally come to the main result of this section, which is a criterion for $L_\infty[1]$ $\rR$-algebras to be linearizable over $\rR$.
Let us first explain what this means.

\begin{definition}\label{turkey}
\begin{enumerate}
\item
An $L_{\infty}[1]$ $\rR$-algebra $(\LL, d_{\LL},q_{2},q_{3},\ldots)$
is called \emph{abelian} if $q_{k}=0$ for all $k\geq 2$.
\item
An $L_{\infty}[1]$ $\rR$-algebra $(\LL, d_{\LL},q_{2},q_{3},\ldots)$
is called \emph{$\rR$-linearizable} if it is $L_{\infty}[1]$ $\rR$-isomorphic
to an abelian one.
\end{enumerate}
\end{definition}

\begin{remark}\label{rem:Fid}
An $L_\infty[1]$ $\rR$-algebra $(\LL,d_{\LL},q_2,q_{3},\ldots)$ is $\rR$-linearizable if and only if there is an $L_\infty[1]$ $\rR$-isomorphism $G:(\LL,d_{\LL},q_2,q_{3},\ldots)\to(\LL,d_{\LL},0,0,\ldots)$ with linear part $g_1=\id_{\LL}$.
One implication is obvious. For the other, given an $L_\infty[1]$ $\rR$-isomorphism $F=(f_1,f_{2},\ldots):(\LL,d_{\LL},q_2,q_{3},\ldots)\to(\MM,d_{\MM},0,0,\ldots)$, its linear part $f_1:(\LL,d_{\LL})\to(\MM,d_{\MM})$ is an isomorphism of DG $\rR$-modules, and the composition
\[ G:= (f_1^{-1},0,0,\ldots)\circ(f_1,f_{2},\ldots):(\LL,d_{\LL},q_2,q_{3},\ldots )\to(\MM,d_{\MM},0,0,\ldots)\to(\LL,d_{\LL},0,0,\ldots)\]
provides the desired $G$.
\end{remark}

\begin{lemma}\label{lem:Sock}
If an $L_{\infty}[1]$ $\rR$-algebra $(\LL,d_{\LL},q_{2},q_{3},\ldots)$ is $\rR$-linearizable, then 
\[ [q_{2}]=0\in H^{1}(\Hom_{\grR}(S_{\grR}^{2}(\LL), \LL), d_{\LL}) \, . \]
\end{lemma}
\begin{proof}
According to Remark~\ref{rem:Fid}, we may assume that $F=(\id_{\LL},f_{2},f_{3},\ldots):(\LL,d_{\LL},q_2,q_{3},\ldots)\to(\LL,d_{\LL},0,0,\ldots)$ is an isomorphism of $L_{\infty}[1]$ $\rR$-algebras. Then, Equation~\eqref{eq:LooMor}, for $n=2$, implies that, for $x,y\in \LL$, 
\[d_{\LL}(f_{2})(x\odot y):=d_{\LL}(f_{2}(x\odot y))-f_{2}(d_{\LL}(x)\odot y + (-1)^{\degree{x}}x\odot d_{\LL}(y))=q_{2}(x\odot y)\, ,\]
hence $[q_{2}]=0$. This completes the proof.
\end{proof}

\begin{remark} The property of being linearizable depends on the base DGCA: if $\gamma:\rS\to\rR$ is a morphism of DGCAs and $(\LL,d_{\LL},q_{2},q_{3},\ldots)$ is an $L_\infty[1]$ $\rR$-algebra, which can also be seen as an $L_\infty[1]$ $\rS$-algebra via the restriction of scalars functor, then it can happen that $(\LL,d_{\LL},q_{2},q_{3},\ldots)$ is $\rS$-linearizable without being $\rR$-linearizable.
We shall see concrete examples of this in the following section.
\end{remark}
\begin{remark} \label{rem:HomotopyAbelian}
More generally, one can define an $L_\infty[1]$ $\rR$-quasi-isomorphism to be an $L_\infty[1]$ $\rR$-morphism whose linear part is a quasi-isomorphism of DG $\rR$-modules, and an $L_\infty[1]$ $\rR$-algebra to be \emph{$\rR$-homotopy abelian} when it can be connected to an abelian $L_\infty[1]$ $\rR$-algebra through a zig-zag of $L_\infty[1]$ $\rR$-quasi-isomorphisms (for instance, this is always the case if the underlying complex is acyclic).
It is well known that when $\rR=(\KK,0)$, an $L_\infty[1]$ $\KK$-algebra is $\KK$-homotopy abelian if and only if it is $\KK$-linearizable; see for instance,~\cite[Lemma 12.7.2]{MR4485797}.
Over a general DGCA base, however, this is no longer the case, and one can find examples of $\rR$-homotopy abelian $L_\infty[1]$ $\rR$-algebras that are not $\rR$-linearizable.
\end{remark}

Given a graded $\grR$-module $\LL$, we denote  by $\ev_{1_{\grR}}$ the evaluation map
\[ \ev_{1_{\grR}}:\coDer_{\grR}(S_{\grR}(\LL))\to \LL:
\psi \mapsto \psi(1_{\grR}) \in \LL \subset S_{\grR}(\LL) \]
for $\psi\in \coDer_{\grR}(S_{\grR}(\LL))$.
This map fits into an exact sequence of graded $\grR$-modules
\begin{equation}\label{splitseq}
0\to \cF^{1}\coDer_{\grR}(S_{\grR}(\LL))\to\coDer_{\grR}(S_{\grR}(\LL))
\xrightarrow{\ev_{1}}\LL\to 0
.\end{equation}

By Lemma~\ref{lem:MC}, an $L_{\infty}[1]$ $\rR$-algebra structure $(\LL, d_{\LL},q_2,q_3,\ldots)$ on $\LL$ corresponds to a DG $\rR$-coalgebra $(S_\grR(\LL),\delta,\Delta)$ where $\delta=\widetilde{d_{\LL}}+\coderQ$.
This induces a DG Lie $\rR$-algebra structure
$(\coDer_{\grR}(S_{\grR}(\LL)),\liederivative{\DQ},[\argument,\argument])$
on $\coDer_{\grR}(S_{\grR}(\LL))$, where
\begin{equation*}
\liederivative{\DQ}:=[\widetilde{d_{\LL}}+\coderQ,-]:
\coDer_{\grR}^{\bullet}(S_{\grR}(\LL))
\to\coDer_{\grR}^{\bullet+1}(S_{\grR}(\LL))
,\end{equation*}
and moreover, $\cF^{1}\coDer_{\grR}(S_{\grR}(\LL))\subset\coDer_{\grR}(S_{\grR}(\LL))$
is a DG Lie $\rR$-subalgebra and the map
$\ev_{1_{\grR}}:(\coDer_{\grR}(S_{\grR}(\LL)),\liederivative{\DQ})
\to(\LL,d_{\LL})$ is a morphism of DG $\rR$-modules.

The following theorem is the main result of this section.

\begin{theorem}\label{thm:main1}
An $L_{\infty}[1]$ $\rR$-algebra $(\LL, d_{\LL},q_2,q_3,\ldots)$ is $\rR$-linearizable if and only if the short exact sequence
\begin{equation}\label{splitseqDG}
0\to(\cF^{1}\coDer_{\grR}(S_{\grR}(\LL)),\liederivative{\DQ})
\to (\coDer_{\grR}(S_{\grR}(\LL)),\liederivative{\DQ})
\xrightarrow{\ev_1}(\LL, d_{\LL})\to 0\end{equation}
splits in the category of DG $\rR$-modules.
\end{theorem}

\begin{proof} We shall say that the $L_\infty[1]$ $\rR$-algebra $(\LL,d_{\LL},\coderQ)$ has the \emph{splitting property over $\rR$} when the short exact sequence \eqref{splitseqDG} splits in the category of DG $\rR$-modules (in the course of the proof the splitting property will always be assumed over $\rR$, so we shall omit to specify it in order to alleviate the exposition).
A first simple but useful observation is that this property is invariant under isomorphisms of $L_\infty[1]$ $\rR$-algebras.
Indeed, an isomorphism $F=(f_{1},f_{2},\ldots):(\LL,d_{\LL},q_{2},q_{3},\ldots)\to(\MM,d_{\MM},r_{2},r_{3},\ldots)$ of $L_\infty[1]$ $\rR$-algebras can be seen as an isomorphism $F:(S_{\grR}(\LL),\widetilde{d_{\LL}}+\coderQ)\to(S_\grR(\MM),\widetilde{d_{\MM}}+\coderR)$ between the corresponding DG $\rR$-coalgebras.
In particular, it induces an isomorphism between the corresponding short exact sequences
\[ \begin{tikzcd}
0 \arrow[r] &  \cF^1\coDer_{\grR}(S_{\grR}(\LL)) \arrow[r] \arrow[d, "\Phi"'] & \coDer_{\grR}(S_{\grR}(\LL)) \arrow[r] \arrow[d, "\Phi"'] & \LL \arrow[r] \arrow[d, "f_1"'] & 0 \\ 0 \arrow[r] &  \cF^1\coDer_{\grR}(S_{\grR}(\MM)) \arrow[r]  & \coDer_{\grR}(S_{\grR}(\MM)) \arrow[r]  & \MM \arrow[r] & 0
\end{tikzcd} \]
where $\Phi(\Xi):=F\Xi F^{-1}$ for all $\Xi\in \coDer_{\grR}(S_{\grR}(\LL))$.
Since the vertical arrows are isomorphisms of DG $\rR$-modules, this shows that if one of the two sequences splits, so does the other.

For the forward implication, it suffices, by invariance under isomorphisms, to show that abelian $L_\infty[1]$ $\rR$-algebras have the splitting property.
Let $(\LL,d_{\LL},0,0,\ldots)$ be an abelian $L_{\infty}[1]$ $\rR$-algebra. 
For homogeneous $x\in\LL$, define $\rho_x:\grR\to\LL$ by $\rho_x(r)=(-1)^{\degree{r}\degree{x}}rx$, and set
\begin{equation}\label{eq:sigma}
\sigma(x)=x\odot-:S_{\grR}(\LL)\to S_{\grR}(\LL) \, .
\end{equation}
It is easy to check that $\sigma(x)$ is an $\grR$-coderivation of $S_{\grR}(\LL)$ with coefficients $p\sigma(x)=(\rho_x,0,0,0,\ldots)$.
It is also clear that \[ \sigma: \LL \to \coDer_{\grR}(S_{\grR}(\LL)):x\to\sigma(x) \] is an $\grR$-linear splitting of the exact sequence \eqref{splitseq}.
Hence it remains to prove that $\sigma$ is compatible with the differentials: this is a straightforward computation, recalling that in the abelian case the differential $\mathcal{L}_\delta$ on $\coDer_{\grR}(S_{\grR}(\LL))$ is given by $\mathcal{L}_\delta=[\widetilde{d_{\LL}},-]$, where $\widetilde{d_{\LL}}:S_{\grR}(\LL)\to S_{\grR}(\LL)$ is the natural extension of $d_{\LL}:\LL\to \LL$.

For the converse implication, let $(\LL,d_{\LL},\coderQ)= (\LL,d_{\LL},q_{2},q_{3},\ldots)$ be an $L_\infty[1]$ $\rR$-algebra satisfying the splitting property.
Starting with $\coderQfloor{2}=\coderQ$, we shall construct inductively a sequence
\begin{equation}\label{chain}
(\LL,d_{\LL},\coderQfloor{2})\xrightarrow{F^2}(\LL,d_{\LL},\coderQfloor{3})\xrightarrow{F^3}\;\;\cdots\;\;\xrightarrow{F^{n-1}}(\LL,d_{\LL},\coderQfloor{n})\xrightarrow{F^n}(\LL,d_{\LL},\coderQfloor{n+1})\xrightarrow{F^{n+1}}\cdots \end{equation}
of $L_\infty[1]$ $\rR$-algebras and $L_\infty[1]$ $\rR$-isomorphisms such that:
\begin{itemize} \item $(\LL,d_{\LL},\coderQfloor{n})=(\LL,d_{\LL},0,\ldots,0,q^n_n,q^n_{n+1},\ldots)$ is an $L_\infty[1]$ $\rR$-algebra with trivial brackets $q^n_i=0$ in arities $1<i<n$;
\item $F^n=(\id_{\LL},0,\ldots,0,f^n_n,0,\ldots)$ is an $L_\infty[1]$ $\rR$-isomorphism with trivial coefficients $f^n_i=0$ in all arities $i\notin\{1,n\}$, and with linear part $f^n_1=\id_{\LL}$. 
\end{itemize}
Once we have done this, the proof of the theorem will be completed.
Indeed, since $F^n: S_{\grR}(\LL)\to S_{\grR}(\LL)$ acts as the identity on $S^{<n}_{\grR}(\LL)$ for all $n$, the infinite composition
\[ G:= \cdots \circ F^n\circ F^{n-1}\circ\cdots\circ F^3\circ F^2 : (\LL,d_{\LL},\coderQ)\to(\LL,d_{\LL},0)\]
makes sense and provides the desired $L_\infty[1]$ $\rR$-isomorphism between $(\LL,d_{\LL},\coderQ)$ and the abelian $L_\infty[1]$ $\rR$-algebra $(\LL,d_{\LL},0)$.

Next, assume inductively we have constructed the chain \eqref{chain} up to $(\LL,d_{\LL},\coderQfloor{n})$.
In order to proceed with the induction, we need to find a degree zero map $\phi:S^n_{\grR}(\LL)\to \LL$ such that
\begin{equation}\label{phi} q^n_n(x_1,\ldots,x_n) = d_{\LL}\Big(\phi(x_1,\ldots,x_n)\Big)-\sum_{i=1}^n(-1)^{|x_1|+\cdots+|x_{i-1}|} \phi(x_1,\ldots,d_{\LL}(x_i),\ldots,x_n).
\end{equation}
Once we have found such a $\phi$, we can define $F^n:S_{\grR}(\LL)\to S_{\grR}(\LL)$ as the isomorphism of $\grR$-coalgebras with coefficients $f^n_1=\id_{\LL}$, $f^n_n=\phi$, $f^n_i=0$ for $i\neq 1,n$, and the coderivation $\coderQfloor{n+1}$ according to $\widetilde{d_{\LL}}+\coderQfloor{n+1}:=F^n \circ (\widetilde{d_{\LL}}+\coderQfloor{n})\circ (F^n)^{-1}$. 
Since $\coderQfloor{n}\in \cF^{2}\coDer_{\grR}(S_{\grR}(\LL))$ and $F^n:S_{\grR}(\LL)\to S_{\grR}(\LL)$ is an isomorphism of $\grR$-coalgebras, we see that $F^n\circ \coderQfloor{n}\circ (F^n)^{-1}\in \cF^{2}\coDer_{\grR}(S_{\grR}(\LL))$. 
Moreover, a simple computation using $\grR$-linearity of the coefficients of $F^n$ and the Leibniz rule for $\widetilde{d_{\LL}}$ shows that  $F^n\circ\widetilde{d_{\LL}}\circ (F^n)^{-1}-\widetilde{d_{\LL}}\in \cF^{2}\coDer_{\grR}(S_{\grR}(\LL))$ as well. 
Thus, \[\coderQfloor{n+1}=F^n\circ(\widetilde{d_{\LL}}+\coderQfloor{n})\circ (F^n)^{-1}-\widetilde{d_{\LL}}\in \cF^{2}\coDer_{\grR}(S_{\grR}(\LL)) \, .\] 
It is now clear by construction that $(\LL,d_{\LL},\coderQfloor{n+1})$ is an $L_\infty[1]$ $\rR$-algebra and $F^n:(\LL,d_{\LL},\coderQfloor{n})\to(\LL,d_{\LL},\coderQfloor{n+1})$ an $L_\infty[1]$ $\rR$-isomorphism. 
Moreover, let $p:S_{\grR}(\LL)\to \LL$ be the natural projection. Since $q_{i}^n=f^{n}_i=0$ for $1<i<n$ and $f^n_1=\id_{\LL}$, by formulas~\eqref{eq:coder} and~\eqref{eq:morph}, for $1<i<n$, we also have
\begin{multline*} q_i^{n+1}(x_1,\ldots,x_i)=p(\widetilde{d_{\LL}}+\coderQfloor{n+1})F^n(x_1\odot\cdots\odot x_i)\\
=pF^n(\widetilde{d_{\LL}}+\coderQfloor{n})(x_1\odot\cdots\odot x_i)=q^n_i(x_1,\ldots,x_i)=0\, .
\end{multline*}
For $i=n$, since $f^n_n=\phi$, we have
\begin{multline*} 
q_n^{n+1}(x_1,\ldots,x_n)+d_{\LL}\big(\phi(x_1,\ldots,x_n)\big) \\ =p(\widetilde{d_{\LL}}+\coderQfloor{n+1})F^n(x_1\odot\cdots\odot x_n) =pF^n(\widetilde{d_{\LL}}+\coderQfloor{n})(x_1\odot\cdots\odot x_n) \\ =q_n^n(x_1,\ldots,x_n)+\sum_{j=1}^n(-1)^{|x_1|+\cdots+|x_{j-1}|}\phi\big(x_1,\ldots,d_{\LL}(x_j),\ldots,x_n\big) \, .
\end{multline*}
Thus, Equation~\eqref{phi} implies $q_n^{n+1}=0$ as well.
Therefore, we have $\coderQfloor{n+1}\in\cF^{n+1}\coDer_{\grR}(S_{\grR}(\LL))$, and we can proceed with the induction.

To construct $\phi$ satisfying~\eqref{phi}, note that $(\LL,d_{\LL},\coderQfloor{n})$ admits the splitting property, since it is isomorphic to $(\LL,d_{\LL},\coderQ)$.
Choose a splitting
\[\sigma : (\LL,d_{\LL})\to(\coDer_{\grR}(S_{\grR}(\LL)),[\widetilde{d_{\LL}}+\coderQfloor{n},-]):x\to\sigma(x)=(\sigma(x)_{0},\sigma(x)_{1},\sigma(x)_{2},\ldots)\]
such that $\sigma(x)_0=\rho_x:\grR\to \LL$ is given by $\sigma(x)_{0}(r)=(-1)^{\degree{r}\degree{x}} rx$ for all homogeneous $x\in \LL$ and $r\in \grR$.
We define $\phi:S^n_{\grR}(\LL)\to \LL$ according to the formula
\begin{equation}\label{phidef}
\phi(x_1,\ldots,x_n) := -\frac{1}{n}\sum_{i=1}^n \sign_{i} \cdot
\sigma(x_i)_{n-1}(x_1,\ldots,\widehat{x_i},\ldots,x_n)\, ,
\end{equation}
where $\sign_{i}=(-1)^{\degree{x_i}(\degree{x_1}+\cdots+\degree{x_{i-1}})}$.
Notice that $\phi$ is indeed $\grR$-multilinear, as $\sigma$ is $\grR$-linear and the coefficients of $\sigma(x)$ are $\grR$-multilinear by assumption. 
Since $q_i^n=0$ for $2\leq i<n$, formula~\eqref{eq:coder} implies that
\begin{multline*}
p(\widetilde{d_{\LL}}+\coderQfloor{n})\sigma(x_i)(x_1\odot\cdots\odot\widehat{x_i}\odot\cdots\odot x_n) \\
= q^n_n(x_i,x_1,\ldots,\widehat{x_i},\ldots,x_n) + d_{\LL}\Big(\sigma(x_i)_{n-1}(x_1,\ldots,\widehat{x_i},\ldots,x_n)\Big)
,\end{multline*}
\begin{multline*}
p\sigma(x_i)(\widetilde{d_{\LL}}+\coderQfloor{n}) (x_1\odot\cdots\odot\widehat{x_i}\odot\cdots\odot x_n) \\
= \sum_{\substack{1\le j\le n \\ j\neq i}}(-1)^{\sum_{\stackrel{1\le h <j}{h\neq i}}|x_h| }\sigma(x_i)_{n-1}(x_1,\ldots,d_{\LL}(x_j),\ldots,\widehat{x_i},\ldots,x_n)
.\end{multline*}
Since $\sigma$ is compatible with the differentials, we have
\begin{multline*}
\sigma(d_{\LL}(x_i))_{n-1}(x_1,\ldots,\widehat{x_i}\ldots,x_n) \\
\begin{aligned}
&=p\sigma(d_{\LL}(x_i))(x_1\odot\cdots\odot\widehat{x_i}\odot\cdots\odot x_n) \\
&= p\left[\widetilde{d_{\LL}}+\coderQfloor{n},\sigma(x_i)\right] (x_1\odot\cdots\odot\widehat{x_i}\odot\cdots\odot x_n) \\
&= q^n_n(x_i,x_1,\ldots,\widehat{x_i},\ldots,x_n) + d_{\LL}\Big(\sigma(x_i)_{n-1}(x_1,\ldots,\widehat{x_i},\ldots,x_n)\Big)
\end{aligned}
\\-(-1)^{|x_i|}\sum_{\substack{1\le j\le n \\ j\neq i}}(-1)^{\sum_{\stackrel{1\le h <j}{h\neq i}}|x_h| }\sigma(x_i)_{n-1}(x_1,\ldots,d_{\LL}(x_j),\ldots,\widehat{x_i},\ldots,x_n) \, . 
\end{multline*}
By graded symmetry of $q_n^n$, this can be rewritten as
\begin{multline}\label{comp} 
q^n_n(x_1,\ldots,x_n) 
= \sign_{i} \cdot \Big( -d_{\LL}\big(\sigma(x_i)_{n-1}(x_1,\ldots,\widehat{x_i},\ldots,x_n)\big)  + \sigma(d_{\LL}(x_i))_{n-1}(x_1,\ldots,\widehat{x_i},\ldots,x_n)
\\
+ (-1)^{|x_i|}\sum_{\substack{1\le j\le n \\ j\neq i}}(-1)^{\sum_{\stackrel{1\le h <j}{h\neq i}}|x_h| }\sigma(x_i)_{n-1}(x_1,\ldots,d_{\LL}(x_j),\ldots,\widehat{x_i},\ldots,x_n) \Big) \, ,
\end{multline}
where $\sign_{i}=(-1)^{\degree{x_i}(\degree{x_1}+\cdots+\degree{x_{i-1}})}$.
Summing Equation~\eqref{comp} over $i=1,\ldots,n$ and using Equation~\eqref{phidef} proves Equation~\eqref{phi} and completes the proof.
\end{proof}

The $\rR$-linearizability condition can be described in terms of a cohomology class.
The short exact sequence~\eqref{splitseqDG} admits a natural splitting of graded $\grR$-modules
$\sigma:\LL\to \coDer_{\grR}(S_{\grR}(\LL))$,
defined by Equation~\eqref{eq:sigma}, although it need not split in the category of DG $\rR$-modules.

Consider the cochain complex
$\big( \Hom_{\grR}\bigl(\LL,\cF^{1}\coDer_{\grR}(S_{\grR}(\LL))\bigr),d \big)$, where
\[ d(f)=\liederivative{\DQ}\circ f -(-1)^{\degree f}f\circ d_{\LL}\]
for homogeneous $f$.
Although $\sigma$ does not take values in $\cF^{1}\coDer_{\grR}(S_{\grR}(\LL))$, its defect
\[ A_\sigma := \liederivative{\DQ}\circ\sigma-\sigma\circ d_{\LL}\] does.
A direct computation shows that $A_\sigma$ is a $1$-cocycle.
Moreover, its cohomology class vanishes precisely when $\sigma$ can be modified by a map taking values in $\cF^{1}\coDer_{\grR}(S_{\grR}(\LL))$ to become a splitting in the category of DG $\rR$-modules.
Thus, the preceding characterization of $\rR$-linearizability yields the following corollary.

\begin{corollary}
An $L_\infty[1]$ $\rR$-algebra $(\LL,d_{\LL},q_2,q_3,\ldots)$ is $\rR$-linearizable if and only if
\[ [A_\sigma]=0 \in H^1\big( \Hom_{\grR}(\LL,\cF^{1}\coDer_{\grR}(S_{\grR}(\LL))),d \big) \, .\]
\end{corollary}

For future application, we note the following corollary of the above theorem.

\begin{corollary}\label{prop:S-split}
Let $\morDGCA: \rS:=(\grS,d_{\grS})\to \rR:=(\grR,d_{\grR})$ be a morphism of \DGCAs, and let $(\LL;d_{\LL},q_2,q_3,\ldots)$ be an $L_{\infty}[1]$ $\rR$-algebra, also seen as an $L_\infty[1]$ $\rS$-algebra via the restriction of scalars functor. 
In this situation, if the exact sequence \eqref{splitseqDG} splits in the category of DG $\rS$-modules, then $(\LL,d_{\LL},q_2,q_3,\ldots)$ is $\rS$-linearizable.
\end{corollary}

\begin{proof}
The natural projections $\nst{\LL}{\grS}{n}\twoheadrightarrow \nst{\LL}{\grR}{n}$ for $n\geq 1$ and the underlying map $\gamma:\grS\to \grR$ for $n=0$ induce inclusions of graded $\grS$-modules
\[ \Hom_{\grR}(\nst{\LL}{\grR}{n},\LL)\hookrightarrow\Hom_{\grS}(\nst{\LL}{\grS}{n},\LL) \]
for all $n\geq 0$, which assemble to an inclusion of graded $\grS$-modules
\[ \coDer_{\grR}(S_{\grR}(\LL))\cong\prod_{n\geq 0}\Hom_{\grR}(\nst{\LL}{\grR}{n},\LL)
\hookrightarrow\prod_{n\geq 0}\Hom_{\grS}(\nst{\LL}{\grS}{n},\LL)
\cong\coDer_{\grS}(S_{\grS}(\LL)):\coderQ\to \coderQ^{\#} .\]
Writing $\delta=\widetilde{d_{\LL}}+\coderQ:S_\grR(\LL)\to S_\grR(\LL)$ and $\delta^\#=\widetilde{d_{\LL}}+\coderQ^\#:S_\grS(\LL)\to S_\grS(\LL)$, we see that
\[ (\coDer_{\grR}(S_{\grR}(\LL)),\liederivative{\delta},[\argument,\argument])
\hookrightarrow (\coDer_{\grS}(S_{\grS}(\LL)),\liederivative{\delta^\#},[\argument,\argument]) \]
is a monomorphism of DG Lie $\rS$-algebras, and in particular a cochain map.
It is straightforward to show that the composition of this monomorphism with a splitting of~\eqref{splitseqDG} in the category of DG $\rS$-modules yields a splitting of
\begin{equation*}
0 \to(\cF^{1}\coDer_{\grS}(S_{\grS}(\LL)),\liederivative{\delta^\#})
\to(\coDer_{\grS}(S_{\grS}(\LL)),\liederivative{\delta^\#})
\xrightarrow{\ev_1}(\LL,d_{\LL})\to 0
\end{equation*}
in the category of DG $\rS$-modules, and so the desired conclusion follows from Theorem~\ref{thm:main1}.
\end{proof}

\section{Kapranov \texorpdfstring{$L_{\infty}[1]$}{L-infinity[1]} algebras associated with DG Lie algebroids}

\subsection{DG Lie algebroids}

In this section, we briefly go through the basic notions on DG Lie algebroids.
Throughout this section, $\KK=\RR$ or $\KK=\CC$.

\subsubsection{The Atiyah class of DG Lie algebroids}

Let $\cM$ be a graded manifold. A graded Lie algebroid over $\cM$
is a graded vector bundle $\cL\to \cM$ equipped with a Lie bracket
$\liecL{\argument}{\argument}$ on $\sections{\cL}$ and a bundle map
$\rho:\cL\to \tangent{\cM}$ of degree $0$, called the anchor map, satisfying the Leibniz rule
\[ \liecL{l_{1}}{f\cdot l_{2}}=(\rho(l_{1})f)\cdot l_{2}+(-1)^{\degree{f}
\cdot \degree{l_{1}}}f\cdot \liecL{l_{1}}{l_{2}} \]
for $f\in\smooth{\cM}$, $l_{1},l_{2}\in\sections{\cL}$.

Let $(\cM,Q)$ be a DG manifold. A DG Lie algebroid over $(\cM,Q)$
is a graded Lie algebroid $\cL\to \cM$ equipped with a differential
$\QL:\sections{\cL}^{\bullet}\to \sections{\cL}^{\bullet +1}$
such that the following compatibility conditions are satisfied:
\begin{enumerate}
\item $\QL(f\cdot l)=Q(f)\cdot l + (-1)^{\degree{f}}f\cdot \QL(l)$
\item $\QL(\liecL{l_{1}}{l_{2}})=\liecL{\QL(l_{1})}{l_{2}}
+(-1)^{\degree{l_{1}}}\liecL{l_{1}}{\QL(l_{2})}$
\item \label{item:QL} $\liederivative{Q}(\rho(l))=\rho(\QL(l))$
\end{enumerate}
for any homogeneous $f\in\smooth{\cM}$ and $l,l_{1},l_{2}\in\sections{\cL}$.
Here, $\liederivative{Q}:\XX(\cM)\to \XX(\cM)$ denotes the Lie derivative along $Q$.

\begin{remark}
In~\cite{MR2534186}, DG Lie algebroids (a.k.a. $Q$-algebroids) are defined
as Lie algebroid objects in the category of DG manifolds. In fact, it is proved
in the same paper that this definition is equivalent to the definition above
and moreover, the compatibility condition~\ref{item:QL} is redundant.
\end{remark}

\begin{example}
\begin{enumerate}
\item
Any differential graded Lie algebra (DGLA) $(\frakg,d, [\argument,\argument])$
forms a DG Lie algebroid $(\cL,\QL)=(\frakg,d)$ over a point $(\cM,Q)=(\{\text{pt}\},0)$.
Conversely, any DG Lie algebroid $(\cL,\QL)$ over a point is a DGLA.
\item
Let $(\cM,Q)$ be a DG manifold. Then the tangent bundle $\tangent{\cM}$
equipped with the Lie derivative $\liederivative{Q}$ along $Q$ forms a DG Lie algebroid
over $(\cM,Q)$, called the tangent DG Lie algebroid of the DG manifold $(\cM,Q)$.
\end{enumerate}
\end{example}

Given a graded Lie algebroid $\cL\to\cM$ and a graded vector bundle $\cE$ over $\cM$,
an $\cL$-connection on $\cE$ is a $\KK$-bilinear map
$\nabla:\sections{\cL}\times\sections{\cE}\to \sections{\cE}$
of degree $0$ satisfying the conditions:
\begin{enumerate}
\item $\nabla_{f\cdot l}e = f\cdot \nabla_{l}e$ ;
\item $\nabla_{l}(f\cdot e) = \rho(l)f \cdot e
+ (-1)^{\degree{l}\cdot \degree{f}} f\cdot \nabla_{l}e$ ,
\end{enumerate}
for any homogeneous $f\in\smooth{\cM}$, $l\in\sections{\cL}$ and $e\in\sections{\cE}$.

When $\cE=\cL$, the torsion of an $\cL$-connection $\nabla$ on $\cL$
is $T^{\nabla}\in\sections{\Lambda^{2}\cL^{\vee}\tensor\cL}$ defined by
\[ T^{\nabla}(l_{1},l_{2})=\nabla_{l_{1}}l_{2}-(-1)^{\degree{l_{1}}\cdot
\degree{l_{2}}} \nabla_{l_{2}}l_{1} - [l_{1},l_{2}]_{\cL}\]
for any homogeneous $l_{1},l_{2}\in\sections{\cL}$.
We say $\nabla$ is torsion-free if $T^{\nabla}=0$.
Torsion-free connections always exist \cite{MR3910470}:
if $\tilde{\nabla}$ is any connection, then
$\nabla=\tilde{\nabla}-\half T^{\tilde{\nabla}}$ is a torsion-free connection.

The curvature of an $\cL$-connection $\nabla$ on $\cL$ is the degree~$0$ element
$\curvature{\nabla} \in \sections{\Lambda^{2}\cL^{\vee}\tensor \End(\cL)}^{0}$ defined by
\[ \curvature{\nabla}(l_{1},l_{2})l_{3}=\nabla_{l_{1}}\nabla_{l_{2}}l_{3}-(-1)^{\degree{l_{1}}
\cdot\degree{l_{2}}}\nabla_{l_{2}}\nabla_{l_{1}}l_{3}-\nabla_{\liecL{l_{1}}{l_{2}}}l_{3} \]
for any homogeneous $l_{1},l_{2},l_{3}\in\sections{\cL}$.

Given a DG Lie algebroid $(\cL,\QL)$ over $(\cM,Q)$,
consider the graded vector bundle $S^2\cL^{\vee}\tensor\cL$ over $\cM$.
There is an induced differential on $\sections{ S^2\cL^{\vee}\tensor\cL}$,
again denoted by $\QL$, by abuse of notation.
Thus, one obtains a cochain complex
$(\sections{ S^2\cL^{\vee}\tensor \cL }^{\bullet}, \QL)$
whose cohomology group is denoted by $H^{\bullet}(\sections{S^2\cL^{\vee}\tensor \cL}, \QL)$.

Given any torsion-free $\cL$-connection $\nabla$ on $\cL$,
consider a $\smooth{\cM}$-bilinear map of degree $1$
\[ \atiyahcocycleQL:\sections{\cL}\times\sections{\cL}\to\sections{\cL} \]
defined by
\[ \atiyahcocycleQL(l_{1},l_{2})=\QL(\nabla_{l_{1}}l_{2})
-\nabla_{\QL(l_{1})}l_{2} -(-1)^{\degree{l_{1}}} \nabla_{l_{1}}\QL(l_{2}) \]
for $l_{1},l_{2}\in\sections{\cL}$.
It is simple to see that $\atiyahcocycleQL$ is indeed a tensor.
Moreover, it is symmetric:
\[\atiyahcocycleQL(l_{1},l_{2})=(-1)^{\degree{l_{1}}\cdot\degree{l_{2}}}
\atiyahcocycleQL(l_{2},l_{1})\]
for any homogeneous $l_{1},l_{2}\in\sections{\cL}$.
Hence it defines a section of the vector bundle $S^{2}\cL^{\vee}\tensor\cL$,
i.e., $\atiyahcocycleQL\in\sections{S^{2}\cL^{\vee}\tensor\cL}$.

Moreover, $\atiyahcocycleQL$ is a $1$-cocycle of the cochain complex
$(\sections{S^2\cL^{\vee}\tensor\cL }^{\bullet},\QL)$,
called the Atiyah $1$-cocycle of $(\cL,\QL)$ associated with $\nabla$.
Its cohomology class
\[ \atiyahclassQL=[\atiyahcocycleQL]\in H^{1}( \sections{ S^2\cL^{\vee}\tensor \cL }, \QL) \]
is called the Atiyah class of the DG Lie algebroid $(\cL,\QL)$.
It is independent of the choice of $\nabla$ and is the obstruction class
to the existence of an $\cL$-connection on $\cL$ compatible with $\QL$.
Indeed, $\atiyahclassQL=0$ if and only if there exists a torsion-free $\cL$-connection
$\nabla$ on $\cL$ such that $\atiyahcocycleQL=0$.

\begin{example}\label{example:AtDGLA}
\begin{enumerate}
\item
If $(\cL,\QL)$ is a DGLA, the Atiyah class of $(\cL,\QL)$ vanishes.
Indeed, since $\smooth{\cM}\cong \KK$, a $\KK$-bilinear map $\nabla$ given by $\nabla_{l_{1}}l_{2}:=\half [l_{1},l_{2}]$ for all $l_{1}, l_{2}\in \sections{\cL}$ defines a torsion-free $\cL$-connection on $\cL$ such that $\atiyahcocycleQL=\QL(\nabla)=0$.
\item
When $(\cL,\QL)=(\tangent{\cM},\liederivative{Q})$ is the tangent bundle
of a DG manifold $(\cM,Q)$, the Atiyah class of $(\cL,\QL)$ is the Atiyah class
of the DG manifold $(\cM,Q)$.
\end{enumerate}
\end{example}

\subsubsection{PBW for DG Lie algebroids}

Let $(\cM,Q)$ be a DG manifold, and let $(\cL,\QL)$ be a DG Lie algebroid over $(\cM,Q)$.
There are two DG $\rR$-coalgebras arising from $(\cL,\QL)$, where $\rR=(\cR,d_{\cR})$ stands for the DGCA $(\smooth{\cM},Q)$, i.e., $\cR=\smooth{\cM}$ and $d_{\cR}=Q$.

The first one is the DG $\rR$-coalgebra $(\ScL,\QS)$ corresponding to the one in Lemma~\ref{lem:MC} with the zero MC element.
More precisely, the comultiplication is defined by Equation~\eqref{eq:Delta} and $\QS:\ScL\to \ScL$ is defined by $\QS(f)=Q(f)$ for $f\in \cR$ and
\[ \QS(l_{1}\odot\cdots\odot l_{n})=\sum_{\sigma\in \shuffle{1}{n-1}} \sign(\sigma)
\cdot \QL(l_{\sigma(1)})\odot l_{\sigma(2)}\odot \cdots \odot l_{\sigma(n)} \]
for $l_{1},\ldots,l_{n}\in\sections{\cL}$, where $\sign(\sigma)$ denotes the Koszul sign.

The other one is a DG $\rR$-coalgebra structure on $\UcL$, the universal enveloping algebra of the graded Lie algebroid $\cL$.
Recall that $\UcL$ is the reduced $\KK$-tensor algebra of $\big(\cR\oplus \sections{\cL}\big)$ modulo $\mathcal{I}$
\[ \UcL=\big(\bigoplus_{n=1}^{\infty}(\cR\oplus\sections{\cL})^{\tensor n}\big) / \mathcal{I} ,\]
where $\mathcal{I}$ is the two-sided ideal generated by the following 4 types of elements
\[
f\tensor g - f\cdot g \, , \qquad   f\tensor l - f\cdot l \, , \qquad l_{1}\tensor l_{2} -(-1)^{\degree{l_{1}}\cdot
\degree{l_{2}}} l_{2}\tensor l_{1} - \liecL{l_{1}}{l_{2}}\,, \qquad  l\tensor f - (-1)^{\degree{f}\cdot \degree{l}} f\tensor l - \rho(l)f
\]
for homogeneous $f,g\in\cR$ and $l,l_{1},l_{2}\in\sections{\cL}$.
The universal enveloping algebra $\UcL$ is a graded left $\cR$-module.

The $\cR$-linear comultiplication $\Delta:\UcL \to \UcL \tensor_{\cR} \UcL$
is defined by $\Delta(1)=1\tensor 1$ and
\begin{equation}\label{eq:ComultUcL}
\Delta(u_{1}\cdot u_{2})=\Delta(u_{1})\cdot\Delta(u_{2})
\end{equation}
for any $u_{1},u_{2}\in\UcL$.
More explicitly, $\Delta$ is characterised by
\[
\Delta(l_{1}l_{2}\cdots l_{n})=\sum_{k=0}^{n} \sum_{\sigma\in \shuffle{k}{n-k}} \sign(\sigma)
\cdot (l_{\sigma(1)}\cdots l_{\sigma(k)}) \tensor (l_{\sigma(k+1)}\cdots l_{\sigma(n)}).
\]
The differential $\QU: \UcL \to \UcL$ is defined as the extension of $\QL$ by the Leibniz rule:
for $f\in \cR$, $\QU(f)=Q(f)$ and for $l_{1}, \ldots, l_{n} \in \sections{\cL}$,
\[ \QU(l_{1}l_{2} \cdots l_{n})=\sum_{k=1}^{n}(-1)^{\degree{l_{1}}+\cdots
+\degree{l_{k-1}}} l_{1} l_{2} \cdots \QL(l_{k}) \cdots l_{n}\, .\]
It is straightforward to check that $(\UcL,\QU)$ forms a DG $\rR$-coalgebra.

\begin{proposition}\label{prop:SULQ}
Given a DG Lie algebroid $(\cL,\QL)$ over a DG manifold $(\cM,Q)$,
both $(\ScL, \QS)$ and $(\UcL, \QU)$ form DG $\rR$-coalgebras.
\end{proposition}

In fact, $\ScL$ and $\UcL$ are isomorphic as graded $\cR$-coalgebras
(but not necessarily as DG $\rR$-coalgebras) via the Poincaré--Birkhoff--Witt (PBW) map.
Let $\nabla$ be an $\cL$-connection on $\cL$.
The PBW map associated with $\nabla$ is the $\cR$-linear map
\begin{equation}\label{eq:PBW}
\pbw^{\nabla}:\ScL \to \UcL
\end{equation}
defined by the following recursive formulas:
\begin{gather*}
\pbw^{\nabla}(f)=f, \quad \forall f\in \cR; \\
\pbw^{\nabla}(l)=l, \quad \forall l\in \sections{\cL}; \\
\pbw^{\nabla}(\mathfrak l)=\frac{1}{n+1}\sum_{k=1}^{n+1} \sign_{k}
\cdot \big( l_{k}\cdot \pbw^{\nabla}(\mathfrak{l}^{\{k\}})
- \pbw^{\nabla}(\nabla_{l_{k}} \mathfrak{l}^{\{k\}})\big)
,\end{gather*}
where $l_1, \ldots, l_{n+1} \in \sections{\cL}$ are homogeneous and $\mathfrak{l} = l_{1}\odot\cdots\odot l_{n+1} \in S^{n+1}_{\cR}(\sections{\cL})$\footnote{\label{footnote1}Contrary to what is stated in~\cite{MR5036937} (and similarly in~\cite{MR4393962, MR3910470}) the PBW map is defined on the space $\ScL$ rather than on the (generally) strictly larger space $\sections{S\cL}$.}. 
Here, the symbol $\mathfrak{l}^{\{k\}}=l_{1}\odot \cdots \hat{l}_{k}\cdots
\odot l_{n+1}\in S^{n}_{\cR}(\sections{\cL})$ denotes the symmetric tensor of $l_{1},\ldots, l_{n+1}$
except $l_{k}$, and $\sign_{k}=(-1)^{\degree{l_{k}}(\degree{l_{1}}+\cdots+\degree{l_{k-1}})}$.

An argument similar to that of~\cite[Theorem~4.3]{MR3910470} (see also \cite{MR5036937}) shows that the PBW map is well-defined and, indeed, an isomorphism of graded $\cR$-coalgebras.

\begin{proposition}
\label{thm:GrPBW}
Let $\cL$ be a graded Lie algebroid over a graded manifold $\cM$.
Upon a choice of an $\cL$-connection $\nabla$ on $\cL$, the map $\pbw^{\nabla}: \ScL \to \UcL$ is an isomorphism of graded $\cR$-coalgebras.
\end{proposition}

\begin{remark}
Given a graded Lie algebroid $\cL$ and an $\cL$-connection $\nabla$ on $\cL$,
the associated PBW isomorphism $\pbw^{\nabla}: \ScL \to \UcL$ can be thought of
as a formal exponential map.

In order to give substance to this assertion, we will work in the simplified setting
of a (non-graded) Lie groupoid $G\rightrightarrows M$
with associated (non-graded) Lie algebroid $\pi: L\lato M$
--- the base field is assumed to be $\RR$.
We denote the source map by $s:G\to M$ and the target map by $t:G\to M$.
Upon a choice of an $L$-connection $\nabla$ on $L$, the exponential map is a bundle map
\[ \begin{tikzcd}
L \arrow[rr, "\exp^{\nabla}"] \arrow[rd, "\pi"'] && G \arrow[ld, "s"] \\ & M &
\end{tikzcd} \]
defined near the zero section.
In particular, the image of the zero section of $L$ under the exponential map
$\exp^{\nabla}$ consists of the unit elements of the groupoid $G$.
Moreover, $\exp^{\nabla}$ is locally diffeomorphic along the zero section of $L$.
When $G=M\times M$, $L=\tangent{M}$ and $\nabla$ is an affine connection on $M$,
the exponential map $\exp^{\nabla}$ is the usual geodesic exponential map on $M$.

The $\pi$-fibrewise differential operators on $L$
evaluated along the zero section of $\pi:L\to M$
can be naturally identified with
the symmetric tensor algebra of the $\smooth{M}$-module of sections $\sections{L}$.
On the other hand, the space of $s$-fibrewise differential operators on $G$
evaluated along the space $M$ of unit elements of the groupoid $G$
can be naturally identified with the universal enveloping algebra $\cU(L)$
of the Lie algebroid $L\lato M$.
Therefore, the exponential map $\exp^{\nabla}$ induces an isomorphism of left $R=\smooth{M}$-modules
\[ (\exp^{\nabla})_{\ast}:S_{R}(\sections{L})\to\cU(L) .\]
It turns out that, in the special case where $\cL\lato \cM$
is an ordinary (i.e., non-graded) Lie algebroid $L\lato M$,
the morphism $\pbw^{\nabla}:\ScL \to \UcL$ is precisely $(\exp^{\nabla})_{\ast}$,
the fiberwise infinite-order jet of the exponential map $\exp^{\nabla}$ along
the zero section --- see~\cite{MR4271478}.
More precisely, given a function $f:G\to\RR$ on the groupoid,
the Taylor series of the restriction of $f\circ\exp^{\nabla}$
to the fiber of $\pi:L\to M$ above a point $m\in M$ is
\[ \sum_{n=0}^{\infty}\sum_{J\in \NN_{r}^{n}} \frac{1}{n!}
\big(\pbw^{\nabla}(\mathfrak{l}_{J})\cdot f\big)(1_{m})\tensor y_{J} \]
where $r$ denotes the rank of the vector bundle $\pi:L\to M$
(i.e., the dimension of the fibre $L_{m}=\pi^{-1}(m)$),
the symbol $\NN_{r}$ denotes the indexing set $\{1,2,\cdots,r\}$,
the set $\{\mathfrak{l}_1,\dots,\mathfrak{l}_r\}$ denotes
local sections of $L$ whose value at $m$ form a basis for $L_{m}$,
the set $\{y_1,\dots,y_r\}$ denotes the dual sections of $L^{\vee}$,
and the multi-index notations $\mathfrak{l}_J$ and $y_J$ stand for
$\mathfrak{l}_{j_1}\odot\cdots\odot\mathfrak{l}_{j_n}$
and $y_{j_1}\odot\cdots\odot y_{j_n}$ respectively if $J=(j_1,\dots,j_n)$.
Here, we view $\pbw^{\nabla}(\mathfrak{l}_{J})$ as a right-invariant
differential operator on $G$.

This point of view ought to carry over to the graded setting.

\end{remark}

\subsection{Kapranov \texorpdfstring{$L_{\infty}[1]$}{L-infinity[1]} \texorpdfstring{$\rR$}{R}-algebras}
One of our main results constructs, from a DG Lie algebroid, an $L_\infty[1]$ algebra over the DG algebra of functions on its base DG manifold. The construction is obtained by conjugating the differential on the universal enveloping algebra along a PBW isomorphism.

\begin{theorem}
\label{prop:QSRN}
Let $(\cL,\QL)$ be a DG Lie algebroid over a DG manifold $(\cM,Q)$, and let $\rR$ be the DGCA $(C^\infty(\cM),Q)$. 
\begin{enumerate}
\item For every torsion-free $\cL$-connection $\nabla$ on $\cL$, the Taylor coefficients of the operator
\[ \QU^{\pbw;\nabla}:= (\pbw^\nabla)^{-1}\circ\QU\circ\pbw^\nabla : \ScL\to\ScL \]
define an $L_\infty[1]$ $\rR$-algebra structure on $\sections{\cL}$.
\item This $L_\infty[1]$ $\rR$-algebra structure on $\sections{\cL}$ satisfies the following properties:
its unary bracket is the differential $\QL$ on $\sections{\cL}$
and its binary bracket is $q_2=-\atiyahcocycleQL\in\sections{S^2(\cL^\vee)\tensor\cL}$,
where $\At^\nabla$ is the Atiyah cocycle associated with $\nabla$.
\item The $L_\infty[1]$ $\rR$-algebra structures arising from different choices of torsion-free $\cL$-connections are mutually isomorphic.
\end{enumerate}
\end{theorem}
Such an $L_\infty[1]$ $\rR$-algebra is called a Kapranov $L_\infty[1]$ $\rR$-algebra on the DG Lie algebroid $(\cL,\QL)$.

\begin{proof}
By Propositions~\ref{prop:SULQ} and~\ref{thm:GrPBW}, 
the pair $(\ScL, \QU^{\pbw;{\nabla}})$
is a DG $\rR$-coalgebra.
Moreover, since $\QU(1_{\cR})=0$ and $\pbw^{\nabla}(1_{\cR})=1_{\cR}$,
we have $\QU^{\pbw;{\nabla}}(1_{\cR})=0$.

According to Lemma~\ref{lem:MC}, the DG $\rR$-coalgebra $(\ScL,\QU^{\pbw; {\nabla}})$
corresponds to an $L_{\infty}[1]$ $\rR$-algebra $(\sections{\cL}; d_{\cL},q_{2},q_{3},\ldots)$.
A direct computation shows that $d_{\cL}=\QL$ and that,
for $l_{1},l_{2}\in \sections{\cL}$,
\[ \QU^{\pbw;{\nabla}}(l_{1}\odot l_{2})=-\atiyahcocycleQL(l_{1}\odot l_{2})
+\QL(l_{1})\odot l_{2} + (-1)^{\degree{l_{1}}}l_{1}\odot \QL(l_{2}) .\]
Thus, we have $q_{2}=-\atiyahcocycleQL$.

Finally, we prove that the structures arising from different torsion-free connections are isomorphic. Let $\widetilde\nabla$ be another torsion-free $\cL$-connection on $\cL$.
By definition of the conjugated differentials, $\pbw^\nabla$ and $\pbw^{\widetilde\nabla}$ are isomorphisms of DG $\rR$-coalgebras with source differentials $\QU^{\pbw;\nabla}$ and $\QU^{\pbw;\widetilde\nabla}$, respectively.
Hence so is
\[ F:=(\pbw^{\nabla})^{-1}\circ \pbw^{\widetilde{\nabla}}:
\big(\ScL, \QU^{\pbw; \widetilde{\nabla}}\big ) \to
\big(\ScL, \QU^{\pbw;{\nabla}} \big) .\]
Thus, $F$ induces an isomorphism of $L_\infty[1]$ $\rR$-algebras.
This concludes the proof.
\end{proof}

\begin{remark}
We first remark that the isomorphism $F$ constructed in the proof has
$\id_{\sections{\cL}}$ as its linear term, since both PBW maps restrict
to the canonical inclusion on $\sections{\cL}$.

We also remark that the sign of the binary bracket can be reversed, provided that the higher brackets are changed accordingly. 
More precisely, for $n\geq 2$, set $\widetilde{q}_{n}=(-1)^{n-1} q_{n}$. Then
\[ (-\id_{\sections{\cL}}, 0, 0, \ldots): (\sections{\cL}; \QL, q_{2}, q_{3},\ldots) \to (\sections{\cL}; \QL, \widetilde{q}_{2}, \widetilde{q}_{3},\ldots) \]
is a strict $L_{\infty}[1]$ $\rR$-algebra isomorphism. In particular, $\widetilde{q}_{2}=-q_{2}$.
\end{remark}

\begin{remark}
The description of the higher multibrackets
$q_{n}$, $n\geq 3$, given in~\cite[Theorem~4.7]{MR4393962} for the case
$\cL=\tangent{\cM}$, extends verbatim to any DG Lie algebroid
$(\cL,\QL)$.
In particular, the brackets $q_{n}$ in Theorem~\ref{prop:QSRN}
are entirely determined by the Atiyah $1$-cocycle $\atiyahcocycleQL$,
the curvature $R^{\nabla}$ and their higher covariant derivatives.
Moreover, when the curvature $R^{\nabla}$ vanishes,
the brackets $q_{n}$, $n\geq 3$, satisfy the recursive formula:
\[ q_{n}(\mathfrak{l})=\frac{1}{n} \sum_{k=1}^{n} \sign_{k}
\Big((-1)^{\degree{l_{k}}} \nabla_{l_{k}} \big(q_{n-1}(\mathfrak{l}^{\{k\}})\big)
- q_{n-1}(\nabla_{l_{k}}\mathfrak{l}^{\{k\}})\Big) \]
where $\sign_{k}=(-1)^{\degree{l_{k}}(\degree{l_{1}}+\cdots+\degree{l_{k-1}})}$
is the Koszul sign, $\mathfrak{l}=l_{1}\odot\cdots\odot l_{n}\in S^{n}_{\cR}(\sections{\cL})$,
and $\mathfrak{l}^{\{k\}}$ denotes the symmetric tensor of all $l_{1},\ldots,l_{n}$
except $l_{k}$.
\end{remark}

\subsubsection{Linearizability over the base DGCA $\rR$}

Theorem~\ref{prop:QSRN} identifies the binary bracket of a Kapranov $L_\infty[1]$ $\rR$-algebra with the Atiyah cocycle, up to sign. It is therefore natural to ask whether the Atiyah class controls the $\rR$-linearizability of this Kapranov $L_{\infty}[1]$ $\rR$-algebra.

For each choice of torsion-free $\cL$-connection $\nabla$ on $\cL$, consider the following map of degree~$1$ \[ C^{\nabla}:=\QU\circ \pbw^{\nabla} - \pbw^{\nabla} \circ \QS:\ScL \to \UcL .\]
Equivalently, $C^{\nabla}=\pbw^{\nabla}\circ \big(\QU^{\pbw;\nabla}-\QS\big)$.
Since $\QS$ is the coderivation determined by $\QL$, this identity shows that $C^\nabla$ records the higher multibrackets of the corresponding Kapranov $L_{\infty}[1]$ $\rR$-algebra.
It is clear that $C^{\nabla}$ is graded $\cR$-linear.
Moreover, it is characterised by the following lemma, whose proof follows from~\cite[Proposition~3.7]{MR4393962}.

\begin{lemma}\label{lem:CNabla}
The $\cR$-linear map $C^{\nabla}$ satisfies
\[
C^{\nabla}(f)=0\, , \qquad 
C^{\nabla}(l)=0\, , \qquad
C^{\nabla}(l_{1}\odot l_{2})= - \atiyahcocycleQL(l_{1}\odot l_{2}) \, ,
\]
for $f\in \smooth{\cM}$ and $l, l_{1}, l_{2} \in \sections{\cL}$,
and for $n\geq 3$, it satisfies the recursive formula
\begin{equation} \label{eq:CNabla}
\begin{multlined}[0.8 \textwidth]
C^{\nabla}(\mathfrak{l})= \frac{1}{n} \sum_{k=1}^{n} \sign_{k}
\Big [ (-1)^{\degree{l_{k}}} l_{k} \cdot C^{\nabla} (\mathfrak{l}^{\{k\}})
- C^{\nabla}(\nabla_{l_{k}}\mathfrak{l}^{\{k\}})\Big]\\
-\frac{2}{n}\sum_{i<j} (-1)^{\degree{l_{i}}\degree{l_{j}}} \sign_{i} \sign_{j} \pbw^{\nabla}
\Big(\atiyahcocycleQL(l_{i}\odot l_{j})\odot \mathfrak{l}^{\{i,j\}} \Big),
\end{multlined}
\end{equation}
where $l_1, \ldots, l_{n} \in \sections{\cL}$ are homogeneous,
$\mathfrak{l}=l_{1}\odot\cdots\odot l_{n}\in S^{n}_{\cR}(\sections{\cL})$
and $\sign_{k}=(-1)^{\degree{l_{k}}(\degree{l_{1}}+\cdots+\degree{l_{k-1}})}$.
Here, the symbol $\mathfrak{l}^{\{k\}}$ denotes the symmetric tensor
of all $l_{1},\ldots, l_{n}$ except $l_{k}$, and similarly
the symbol $\mathfrak{l}^{\{i,j\}}$ denotes the symmetric tensor
of all $l_{1},\ldots, l_{n}$ except $l_{i}$ and $l_{j}$.
\end{lemma}

The following theorem provides equivalent conditions for $\rR$-linearizability of this particular type of Kapranov $L_\infty[1]$ $\rR$-algebras.

\begin{theorem}\label{thm:DGPBW}
Assume that $(\cL, \QL)$ is a DG Lie algebroid over $(\cM, Q)$.
The following statements are equivalent.
\begin{enumerate}
\item Every Kapranov $L_{\infty}[1]$ $\rR$-algebra on $(\sections{\cL},\QL)$
is $\rR$-linearizable. \label{item:Linear}
\item The Atiyah class of the DG Lie algebroid $(\cL, \QL )$ vanishes. \label{item:At0}
\item There exists a torsion-free $\cL$-connection $\nabla$ on $\cL$ satisfying
$\QU\circ\pbw^\nabla=\pbw^\nabla\circ\QS$.
\label{item:AtPBW}
\end{enumerate}
\end{theorem}

\begin{proof}
First, \ref{item:Linear} $\Rightarrow$ \ref{item:At0}
follows directly from Lemma~\ref{lem:Sock}. 

Next, we show \ref{item:At0} $\Rightarrow$ \ref{item:AtPBW}.
Assume that the Atiyah class vanishes. Then there exists a torsion-free
$\cL$-connection $\nabla$ on $\cL$ such that $\atiyahcocycleQL=0$.
By Lemma~\ref{lem:CNabla}, $C^{\nabla}(l_{1}\odot l_{2})=0$
for all $l_{1},l_{2}\in \sections{\cL}$.
By the recursive formula~\eqref{eq:CNabla}, one can inductively show that
$C^{\nabla}= 0$. By the definition of $C^{\nabla}$,
the assertion~\ref{item:AtPBW} holds.

Finally, we show \ref{item:AtPBW} $\Rightarrow$ \ref{item:Linear}. 
Assume that $\QU\circ\pbw^\nabla=\pbw^\nabla\circ\QS$. By construction, we have $\QU^{\pbw;{\nabla}}=\QS$.
That is, the linear coefficient of $\QU^{\pbw;{\nabla}}$ is $\QL$
and the higher coefficients of $\QU^{\pbw;{\nabla}}$ all vanish.
By Theorem~\ref{prop:QSRN}, the Kapranov $L_{\infty}[1]$ $\rR$-algebra structure on $(\sections{\cL},\QL)$
associated with the chosen connection $\nabla$ is abelian, which implies assertion~\ref{item:Linear}.
This completes the proof.
\end{proof}

\begin{example}
If $(\cL,\QL)$ is a DGLA $(\frakg,d)$, its Atiyah class necessarily vanishes (cf. Example~\ref{example:AtDGLA}).
Thus, there exists a torsion-free $\cL$-connection $\nabla$ on $\cL$ that
induces an isomorphism
\[\pbw^{\nabla}:(\ScL,\QS) \to (\UcL,\QU)\]
of DG $\KK$-coalgebras.
Indeed, the canonical torsion-free connection
$\nabla=\half \lie{\argument}{\argument}$ 
given by half of the Lie bracket on $\frakg$ induces the classical PBW isomorphism $\pbw^{\nabla}$ of DG $\KK$-coalgebras.
\end{example}

\subsubsection{Linearizability over the base field $\KK$}
We now turn from $\rR$-linearizability to linearizability over $\KK$. 
As noted in Remark~\ref{rem:UnderlyingK}, restriction of scalars turns a Kapranov $L_\infty[1]$ $\rR$-algebra $\sections{\cL}$ constructed in Theorem~\ref{prop:QSRN} into an $L_\infty[1]$ $\KK$-algebra with the same multibrackets viewed as $\KK$-multilinear maps.
Although this algebra need not be linearizable over $\rR$, the situation over $\KK$ is different.
\begin{theorem}\label{thm:LAHA}
For every DG Lie algebroid $\cL$, the $L_\infty[1]$ $\KK$-algebra $\sections{\cL}$ obtained from any of its Kapranov $L_\infty[1]$ $\rR$-algebras by restriction of scalars is necessarily $\KK$-linearizable.
\end{theorem}

\begin{proof}
Let $(\cL, \QL)$ be a DG Lie algebroid over a DG manifold $(\cM,Q)$,
and let $\nabla$ be a torsion-free $\cL$-connection on $\cL$.
According to Corollary~\ref{prop:S-split} with $\rS=\KK$,
it suffices to show that there exists a right inverse $\tau$ of
$\ev_{1_{\grR}}:(\coDer_{\grR}(\ScL), \liederivative{\QU^{\pbw}})\to (\sections{\cL},\QL)$
in the category of DG $\KK$-modules.

Define a DG $\KK$-module homomorphism
$\tau:(\sections{\cL},\QL) \to (\coDer_{\grR}(\ScL), \liederivative{\QU^{\pbw}})$ by
\begin{equation*}
\tau(\eta): \mathfrak{l}\mapsto (-1)^{\degree{\eta}\cdot
\degree{\mathfrak{l}}} (\pbw^{\nabla})^{-1}\big( \pbw^{\nabla}(\mathfrak{l})\cdot \eta \big)
\end{equation*}
for $\eta\in \sections{\cL}$ and $\mathfrak{l}\in \ScL$.
We claim that $\tau$ is indeed well-defined and is a DG right inverse of $\ev_{1_{\grR}}$.

Note that $\tau$ is obtained by conjugating the map
$\tilde{\tau}:(\sections{\cL},\QL) \to (\coDer_{\grR}(\UcL), \liederivative{\QU})$,
defined by
\[ \tilde{\tau}(\eta) : u \mapsto (-1)^{\degree{u}\cdot \degree{\eta}} u \cdot \eta \]
for $\eta\in \sections{\cL}$ and $u\in \UcL$, by the map $\pbw^{\nabla}$.
By construction, the map $\pbw^{\nabla}:(\ScL,\QU^{\pbw;{\nabla}})\to (\UcL,\QU)$
is an isomorphism of DG $\rR$-coalgebras.
It thus suffices to show that $\tilde{\tau}$ is a cochain map.
It is clear that each operator $\tilde{\tau}(\eta)$ is graded $\grR$-linear.
Moreover, by Equation~\eqref{eq:ComultUcL}, we have
\[\Delta(u\cdot \eta)=\Delta(u) \cdot (\eta\tensor 1_{\grR} + 1_{\grR}\tensor \eta),
\qquad \forall\eta\in \sections{\cL},\quad \forall u\in \UcL .\]
Hence $\tilde{\tau}(\eta)\in\coDer_{\grR}(\UcL)$.
The equation
\[\QU(u\cdot \eta)=\QU(u)\cdot \eta + (-1)^{\degree{u}}u\cdot \QL(\eta),
\qquad \forall \eta\in \sections{\cL},\quad \forall u\in \UcL\]
implies that $\tilde{\tau}$ is indeed compatible with differentials.
Finally, it is easy to check that $\ev_{1_{\grR}}\circ\tau(\eta)=\eta$.
This completes the proof.
\end{proof}

Specializing the preceding results to the tangent DG Lie algebroid $(\tangent{\cM},\liederivative Q)$ of a DG manifold $(\cM,Q)$ yields the following corollary, which also summarizes the main conclusions of this section.
\begin{corollary}
\label{cor:DGManHA}
Let $(\cM,Q)$ be a DG manifold and let $\rR=(\smooth{\cM},Q)$. Then:
\begin{enumerate}
\item \label{item:TM1} the tangent DG Lie algebroid $(\tangent{\cM},\liederivative Q)$ induces a Kapranov $L_\infty[1]$ $\rR$-algebra structure on $(\XX(\cM),\liederivative Q)$ whose binary bracket represents, up to sign, the Atiyah class of $(\cM,Q)$;
\item \label{item:TM2} this Kapranov $L_\infty[1]$ $\rR$-algebra is $\rR$-linearizable if and only if the Atiyah class of $(\cM,Q)$ vanishes;
\item \label{item:TM3} the $L_\infty[1]$ $\KK$-algebra obtained from this Kapranov $L_{\infty}[1]$ $\rR$-algebra by restriction of scalars is always $\KK$-linearizable (equivalently, $\KK$-homotopy abelian).
\end{enumerate}
\end{corollary}

Part~\ref{item:TM1} of Corollary~\ref{cor:DGManHA} was previously established in~\cite{MR4393962}.

\begin{remark}
The distinction between the base DGCAs in Corollary~\ref{cor:DGManHA} is crucial: part~\ref{item:TM2} concerns linearizability over $\rR=(\smooth{\cM},Q)$, whereas part~\ref{item:TM3} concerns linearizability over the base field $\KK$, regarded as the DGCA $(\KK,0)$.
\end{remark}

\begin{remark}
For the DG manifold $(\frakg[1],d_{\CE})$ associated with an $L_\infty$ algebra $\frakg$,
since $\XX(\frakg[1])\cong\widehat{S}(\frakg^\vee[-1])\otimes\frakg[1]$,
there is an induced Kapranov $L_\infty$ $\rR$-algebra structure on $\widehat{S}(\frakg^\vee[-1])\otimes\frakg$,
where $\rR$ denotes the DG algebra $\big(\widehat{S}(\frakg^\vee[-1]),d_{\CE}\big)$.
Therefore, our Corollary~\ref{cor:DGManHA}\ref{item:TM3} recovers a result of Lazarev \cite[Theorem~6.4]{MR3237735}.
\end{remark}

\begin{remark}
\label{rem:Difference}
As noted in Remark~\ref{rem:HomotopyAbelian}, when $\rR=\KK$ is the base field, $\KK$-homotopy abelianness and $\KK$-linearizability are equivalent.
However, the situation is different for an arbitrary \DGCA{} $\rR$; it follows immediately from Definition~\ref{turkey} and Remark~\ref{rem:HomotopyAbelian} that $\rR$-linearizability implies $\rR$-homotopy abelianness, but the converse is not true.
The following example provides a counterexample.

Consider the DG manifold $(\cM,Q)=(\frakg[1],d_{\CE})$ associated with the two-dimensional graded Lie algebra $\frakg=\KK x \oplus \KK y$ with a generator $x$ of degree $0$, a generator $y$ of degree $+1$, and the commutation relations \[ [x,x]=0, \qquad [x,y]=y, \qquad [y,y]=0\, . \]
Its tangent DG Lie algebroid $(\tangent{\cM},\liederivative{Q})$ induces a Kapranov $L_\infty[1]$ $\rR$-algebra structure on $(\XX(\cM),\liederivative{Q})$ whose binary bracket is, up to sign, a cocycle representative of the Atiyah class.
Here $\rR$ denotes the DG algebra $(\smooth{\cM},Q)$ of functions on $(\cM,Q)=(\frakg[1],d_{\CE})$.

A direct computation shows that the cochain complex $(\XX(\cM),\liederivative{Q})$ is acyclic.
As noted in Remark~\ref{rem:HomotopyAbelian}, any $L_\infty[1]$ $\rR$-algebra with acyclic underlying cochain complex is necessarily $\rR$-homotopy abelian: indeed, its unique strict morphism to the zero $L_\infty[1]$ $\rR$-algebra is a (obviously graded $\cR$-linear) quasi-isomorphism.
Hence the Kapranov $L_\infty[1]$ $\rR$-algebra on $(\XX(\cM),\liederivative{Q})$ is $\rR$-homotopy abelian.
On the other hand, one checks that the Atiyah class of $(\cM,Q)$ is nonzero.
Therefore, by Corollary~\ref{cor:DGManHA}, this Kapranov $L_\infty[1]$ $\rR$-algebra is not $\rR$-linearizable.
\end{remark}

\subsection{Application to Lie pairs}
\label{sec:DGLAdLP}
Lie pairs provide another example to which the previous results apply.
On the one hand, every Lie pair gives rise to a DG Lie algebroid and hence, by Theorem~\ref{prop:QSRN}, to a Kapranov $L_\infty[1]$ $\rR$-algebra.
On the other hand, a Kapranov $L_\infty[1]$ $\rR$-algebra can be constructed directly from the Lie pair~\cite{MR4271478}.
In this subsection, we recall both constructions and prove that the two resulting $L_\infty[1]$ $\rR$-algebras are indeed quasi-isomorphic.

\subsubsection{Construction via DG Lie algebroids}

Let $M$ be a smooth manifold.
A Lie pair over $M$ is a pair $(L,A)$ of Lie algebroids over $M$ such that $A$ is a Lie subalgebroid of $L$.

\begin{example}\label{ex:foliation}
Let $M$ be a smooth manifold and $\cF$ be a regular foliation on $M$.
Then the integrable distribution $\tangent{\cF}\subset \tangent{M}$ of $\cF$
is a Lie subalgebroid of $\tangent{M}$.
Thus, $(\tangent{M},\tangent{\cF})$ is a Lie pair.
Similarly, when $X$ is a complex manifold and $\tangentzo{X} \subset \tangentcc{X}$
is the anti-holomorphic tangent bundle, $(\tangentcc{X}, \tangentzo{X})$ is a Lie pair.
\end{example}

We begin by recalling the DG Lie algebroid associated with a Lie pair. For details, see~\cite{MR5036937}.

Consider the DG manifold $(\cM,Q)=(A[1],d_{\CE})$.
The graded algebra of functions $C^\infty (\cM)$
is $\Omega_{A}^{\bullet}:=\sections{\Lambda^{\bullet} A^{\vee}}$,
and the homological vector field $Q$ is the Chevalley--Eilenberg differential
$d_{\CE}$ of the Lie algebroid $A$.

The graded Lie algebroid $\cL$ is the pull-back Lie algebroid $\cL=\pi^{!}L$
of $L\to M$ via the morphism of graded manifolds $\pi:A[1]\to M$.
More precisely, let $\rho_{L}:L\to \tangent{M}$ be the anchor map
of $L$ and $\tangent{\pi}:\tangent{A[1]}\to \tangent{M}$ be the tangent map of $\pi$.
The pull-back vector bundle $\pi^{!}L=\tangent{A[1]}\times_{\tangent{M}}L$
of $\rho_{L}$ over $\tangent{\pi}$
\[ \begin{tikzcd}
\pi^{!}L \arrow{r} \arrow{d}& L \arrow{d}{\rho_{L}}\\
\tangent{A[1]}\arrow{r}{\tangent{\pi}} & \tangent{M}
\end{tikzcd} \]
is a graded Lie algebroid over $\cM$. The graded space
of sections $ \sections{\pi^{!}L}$ consists of
elements of the form $(X,\sum_{i} f_{i}\tensor \xi_{i})\in
\XX(A[1])\times \Omega_{A}(L)$, with $f_{i}\in \Omega_{A}$
and $\xi_{i}\in \sections{M; L}$, satisfying
\[\tangent{\pi}(X)=\sum_{i}f_{i}\tensor \rho_{L}(\xi_{i})
\in \Omega_{A}(\tangent{M})\cong \sections{A[1]; \pi^{\ast}TM}\, . \]
The anchor map $\rho_{\pi^{!}L}$ is the projection to $\XX(A[1])$,
and the Lie bracket on $\sections{\pi^{!}L}$ is defined by
\begin{multline*}
\liecL{(X, \sum_{i}f_{i}\tensor \xi_{i})}{(Y, \sum_{j}g_{j}\tensor \eta_{j})} = \\
= (\lie{X}{Y}, \sum_{j}X(g_{j})\tensor \eta_{j}
-  \sum_{i} (-1)^{\degree{Y}\cdot \degree{f_{i}}} Y(f_{i}) \tensor \xi_{i}
+ \sum_{i,j} f_{i}\cdot g_{j}\tensor \lie{\xi_{i}}{\eta_{j}})
\end{multline*}
for $X,Y\in \XX(A[1])$, $f_{i},g_{j}\in \Omega_{A}$ and $\xi_{i},\eta_{j} \in \sections{M; L}$.

Finally, let $s=(d_{\CE},\iota)\in\sections{\pi^{!}L}$,
where $\iota\in\Omega_A^1(L)$ corresponds to the inclusion $A\into L$.
More explicitly, $\iota$ admits a local expression
\[\iota=\sum_{i} \xi^{i}\tensor e_{i} \in \sections{A^{\vee} \tensor A}\subset \Omega_{A}(L) \, ,\]
where $(e_1,\ldots,e_m)$ denotes a local frame for $A$, and $(\xi^{1},\ldots,\xi^{m})$ denotes the dual frame for $A^{\vee}$.

It is simple to see that $s$ is indeed a degree $1$ section of $\cL=\pi^{!}L$ satisfying
$\rho_{\cL}(s)=d_{\CE}$ and $\liecL{s}{s}=0$.
Thus, it defines a DG Lie algebroid $(\cL,\QL)=(\pi^{!}L, \liederivative{s})$
over $(\cM,Q)=(A[1],d_{\CE})$, where $\liederivative{s}:=\liecL{s}{\argument}$
is the Lie derivative along $s\in\sections{\cL}$.
According to Theorem~\ref{prop:QSRN}, we obtain

\begin{theorem}\label{thm:Liepair1}
Given a Lie pair $(L,A)$, there is an induced $L_{\infty}[1]$ $\rR$-algebra structure on $\sections{A[1]; \pi^{!}L}$, where $\rR$ denotes the DGCA $(\Omega_{A}^{\bullet},d_{\CE})$.
Its unary bracket is $\liederivative{s}$, and its binary bracket is, up to sign, a cocycle representative of the Atiyah class of the DG Lie algebroid $(\pi^{!}L, \liederivative{s})$.
\end{theorem}

\subsubsection{Direct construction from a Lie pair}
We now recall the ingredients of the direct construction of a Kapranov $L_\infty[1]$ $\rR$-algebra associated with a Lie pair $(L,A)$ that will be needed for the comparison below.
Here $\rR=(\Omega_A^\bullet,d_{\CE})$.
For details, see~\cite{MR4271478}.

Given a Lie pair $(L,A)$ over a manifold $M$, set $B=L/A$ and $R=\smooth{M}$.
Consider the exact sequence of vector bundles
\begin{equation}\label{eq:LPSES}
0\to A\into L \xto{q} B \to 0.
\end{equation}
There is a canonical flat $A$-connection
$\nabla^{\Bott}:\sections{A}\times \sections{B}\to \sections{B}$ on $B$,
known as the Bott $A$-connection, defined by
\[ \nabla^{\Bott}_{a} \big(q(l)\big) = q\big([a,l]_{L}\big), \qquad \forall a\in \sections{A},
\quad \forall l\in \sections{L} .\]

Recall that given a Lie algebroid $L$, there is a left $R$-coalgebra $\cU(L)$, the universal enveloping algebra of $L$.
Consider the quotient space $\cUB:=\cU(L)/\big(\cU(L)\cdot \sections{A}\big)$ of $\cU(L)$ by the left ideal $\cU(L)\cdot \sections{A}$ generated by $\sections{A}$.
The left ideal $\cU(L)\cdot \sections{A}$ is, in fact, a coideal, so the quotient space $\cUB$ inherits a left $R$-coalgebra structure. 

The left coalgebras $S_{R}(\sections{B})$ and $\cUB$ are isomorphic as $R$-coalgebras via a PBW isomorphism\footnote{Contrary to what is stated in~\cite{MR4271478}, the PBW map is defined on 
$S_{R}(\sections{B})$, rather than the generally (strictly) larger space $\sections{SB}$.}.
\begin{theorem}[\cite{MR4271478}]\label{thm:LAPBW}
Upon a choice of a splitting $j:B\to L$ of exact sequence~\eqref{eq:LPSES} and a choice of an $L$-connection $\nabla^{B}$ on $B=L/A$ extending the Bott $A$-connection, there exists a unique isomorphism of $R$-coalgebras
\[\underline{\pbw}=\underline{\pbw}^{\nabla^{B},j}:S_{R}(\sections{B}) \to \cUB\]
satisfying
\begin{gather}
\underline{\pbw}(f)=f, \quad \forall f\in \smooth{M} ; \nonumber \\
\underline{\pbw}(b)=j(b), \quad \forall b\in \sections{B} ; \nonumber \\
\underline{\pbw}(b_{1}\odot \cdots \odot b_{n})=\frac{1}{n}\sum_{i=1}^{n}
\big(j(b_{i})\centerdot \underline{\pbw}(\mathbf{b}^{\{i\}})
-\underline{\pbw}(\nabla^{B}_{j(b_{i})}\mathbf{b}^{\{i\}})\big) \label{complimentary}
\end{gather}
for all $b_{1},\cdots,b_{n}\in \sections{B}$, where the symbol $\centerdot$ on the RHS of Equation~\eqref{complimentary} denotes the left $\cU(L)$-module structure on $\cUB$ and the notation $\mathbf{b}^{\{i\}} \in S_{R}(\sections{B})$ denotes the symmetric tensor of $b_{1},\cdots,b_{n}$ except $b_{i}$.
\end{theorem}

Note that $\Omega_{A}^{\bullet}$ is a graded $R$-algebra
and that the $R$-coalgebra structures on $S_{R}(\sections{B})$
and $\cUB$ extend to graded $\Omega_{A}$-coalgebras 
$S_{\Omega_{A}}(\Omega_{A}(B))\cong \Omega_{A}\tensor_{R}S_{R}(\sections{B})$ and $\Omega_{A}\tensor_{R}\cUB$, respectively.
Clearly, the above PBW isomorphism naturally extends to
\begin{equation}
\label{eq:PBWliepair}
\underline{\pbw}:S_{\Omega_{A}}(\Omega_{A}(B)) \to \Omega_{A}^{\bullet}\tensor_{R}\cUB,
\end{equation}
an isomorphism of graded $\Omega_{A}$-coalgebras.

Left multiplication gives $\cUB$ an $A$-module structure.
Let $d_A$ be the associated Chevalley--Eilenberg differential on $\Omega_A\tensor_R\cUB$.
Conjugating this differential along the PBW isomorphism \eqref{eq:PBWliepair}, we obtain
\begin{equation*}
d^\lightning:=(\underline{\pbw})^{-1}\circ d_A\circ
\underline{\pbw}:
S_{\Omega_A}(\Omega_A(B))\to S_{\Omega_A}(\Omega_A(B)).
\end{equation*}
This makes $(S_{\Omega_A}(\Omega_A(B)),d^\lightning)$ a DG $\rR$-coalgebra.
Moreover, $d^\lightning(1_{\Omega_A})=0$, and its restriction to $\Omega_A(B)$ is the Chevalley--Eilenberg differential $d^{\Bott}$ induced by the Bott $A$-connection.

Thus, according to Lemma~\ref{lem:MC}, the DG $\rR$-coalgebra
$(S_{\Omega_{A}}(\Omega_{A}(B)), d^{\lightning})$ induces an $L_{\infty}[1]$ $\rR$-algebra
$(\Omega_{A}(B),d^{\Bott},\widetilde{q}_{2},\widetilde{q}_{3},\ldots)$, where
\[ \widetilde{q}_{n}\in \Hom^{1}_{\Omega_{A}}(S^{n}_{\Omega_{A}}(\Omega_{A}(B)),
\Omega_{A}(B))\cong\Omega_{A}^{1}(S^{n}B^{\vee}\tensor B) \]
for $n\geq 2$. Assume that the $L$-connection $\nabla^{B}$ on $B$
in Theorem~\ref{thm:LAPBW} is torsion-free, i.e.,
\[ \nabla^{B}_{l_{1}}q(l_{2}) - \nabla^{B}_{l_{2}}q(l_{1})= q(\lie{l_{1}}{l_{2}}) \]
for all $l_{1},l_{2}\in\sections{L}$.
Then $\widetilde{q}_{2}\in\Omega_{A}^{1}(S^{2}B^{\vee}\tensor B)$ satisfies
\[- \widetilde{q}_{2}(a; b_{1}\odot b_{2})=\nabla^{B}_{a}\nabla^{B}_{j(b_{1})}b_{2}
-\nabla^{B}_{j(b_{1})}\nabla^{B}_{a}b_{2} - \nabla^{B}_{\lie{a}{j(b_{1})}_{L}}b_{2} \]
for all $a\in\sections{A}$ and $b_{1},b_{2}\in\sections{B}$.
In this case, the element $(-\widetilde{q}_{2})$ is a $1$-cocycle in
the cochain complex $(\Omega_{A}^{\bullet}(S^{2}B^{\vee}\tensor B),d^{\Bott})$,
called the Atiyah 1-cocycle associated with $\nabla^{B}$.
Its cohomology class
\[ \alpha_{(L,A)}=[-\widetilde{q}_{2}] \in H_{\CE}^{1}(A; S^{2}B^{\vee}\tensor B) \]
does not depend on the choice of $\nabla^{B}$ and $j$, and is called the Atiyah class
of the Lie pair $(L,A)$ --- see~\cite{MR4271478} for details.

\begin{theorem}[\cite{MR4271478}]\label{thm:LAKapranov}
Let $(L,A)$ be a Lie pair and $B=L/A$ be the quotient bundle.
A torsion-free $L$-connection $\nabla^{B}$ on $B$ extending
the Bott $A$-connection and a splitting $j:B\to L$ of \eqref{eq:LPSES} determine an
$L_{\infty}[1]$ $\rR$-algebra $(\Omega_{A}(B),d^{\Bott},
\widetilde{q}_{2},\widetilde{q}_{3},\ldots)$,
where $\rR$ denotes the DGCA $(\Omega_{A}^{\bullet}, d_{\CE})$.
\begin{enumerate}
\item The unary bracket of this $L_{\infty}[1]$ $\rR$-algebra is the Chevalley--Eilenberg differential
$d^{\Bott}:\Omega^{\bullet}_{A}(B)\to \Omega^{\bullet+1}_{A}(B)$ induced by the Bott $A$-connection on $B$.
\item The binary bracket $\widetilde{q}_{2}: \Omega^{\bullet}_{A}(S^{2}B) \to \Omega^{\bullet+1}_{A}(B)$
is, up to sign, the Atiyah 1-cocycle associated with $\nabla^{B}$ and $j$.
\end{enumerate}
Furthermore, the $L_{\infty}[1]$ $\rR$-algebra structures on $\Omega_{A}(B)$
arising from different choices of $\nabla^{B}$ and $j:B\to L$ are all isomorphic.
\end{theorem}
Such an $L_{\infty}[1]$ algebra is called a Kapranov $L_{\infty}[1]$ $\rR$-algebra of the Lie pair $(L, A)$.
Indeed, for the Lie pair $(\tangentcc{X},\tangentzo{X})$ arising from a Kähler manifold $X$, this construction recovers the $L_{\infty}[1]$ structure on $\Omega_{X}^{0,\bullet}(T_{X})$ discovered by Kapranov \cite{MR1671737}.

\subsubsection{Comparison of the two constructions}
Given a Lie pair $(L,A)$, Theorems~\ref{thm:Liepair1} and~\ref{thm:LAKapranov} yield two Kapranov $L_\infty[1]$ $\rR$-algebras.
We now prove that they are quasi-isomorphic.

\begin{theorem}\label{thm:LPHA}
Let $(L,A)$ be a Lie pair, and let $(\pi^{!}L,\liederivative{s})$ be its corresponding DG Lie algebroid.
Then a Kapranov $L_{\infty}[1]$ $\rR$-algebra $\sections{\pi^{!}L}$ of the DG Lie algebroid $(\pi^{!}L, \liederivative{s})$, as in Theorem~\ref{thm:Liepair1}, and a Kapranov $L_{\infty}[1]$ $\rR$-algebra $\Omega_A(B)$ of the Lie pair $(L, A)$, as in Theorem~\ref{thm:LAKapranov}, are $\rR$-quasi-isomorphic.
\end{theorem}
\begin{proof}

Introduce an $\Omega_A$-linear operator
\[ P_0: \sections{\pi^! L} \to \Omega_{A}(B) \]
by the relation
\[ P_0 (X,\sum_{i} f_{i}\tensor \xi_{i})=\sum_{i} f_{i}\otimes q(\xi_{i}) \]
for any element
$(X,\sum_{i} f_{i}\tensor \xi_{i})\in\XX(A[1])\times \Omega_{A}(L)$
satisfying $\tangent{\pi}(X)=\sum_{i}f_{i}\cdot\rho_{L}(\xi_{i})$.
According to~\cite[Theorem~3.1]{MR5036937}, $P_0$ is the projection of a contraction from $(\sections{\pi^{!}L}, \liederivative{s})$ to $(\Omega_{A}(B), d^{\Bott})$, hence is a quasi-isomorphism.

Following \cite[Lemma~4.5]{MR5036937}, let
\[ P_{\cU} : \cU(\pi^! L)
\to \Omega_{A}\tensor_{R}\cUB \]
be the $\Omega_A$-linear map defined by
\begin{equation}\label{eq:P_U}
P_{\cU}\big((Y_1,v_1\circ\pi) \circc\cdots\circc (Y_n,v_n\circ\pi)\big)
=\class{v_1 \circc \cdots \circc v_n}
\end{equation}
for all $Y_i\in \XX (A[1])$ and $v_i\in \sections{L}$ such
that $T\pi (Y_i)=\rho_L (v_i)$, $\forall i=1, \cdots, n$.
The notation $\overline{u}$ refers to the equivalence class
in $\cUB$ of an element $u\in\cU(L)$.
From~\cite[Lemma~4.5]{MR5036937},
it follows that $P_{\cU}$ is a well-defined
morphism of DG modules
\[ P_{\cU} : \big(\cU(\pi^! L), \QU\big)\to
\big(\Omega_{A}\tensor_{R}\cUB, d_A\big) \]
over the DGCA $(\Omega_A, d_{\CE})$.
Moreover, from Equation~\eqref{eq:P_U},
it is simple to see that $P_{\cU}$ respects the $\Omega_A$-coalgebra
structures. Therefore,
$P_{\cU}$ is indeed a morphism of
DG $\rR$-coalgebras, where $\rR$ is the DGCA $(\Omega_A,d_{\CE})$.

Finally, note that, for each choice of connections and splitting,
both $\pbw^{\nabla}:S_{\Omega_{A}}(\sections{\pi^{!}L}) \to \cU(\pi^{!}L)$ as in~\eqref{eq:PBW} and
$\underline{\pbw}:S_{\Omega_{A}}(\Omega_{A}(B)) \to \Omega_{A}\tensor_{R}\cUB$
as in~\eqref{eq:PBWliepair} are
isomorphisms of graded $\Omega_A$-coalgebras.
It thus follows that
\begin{equation*}
P_{\cU}^{\pbw}:=(\underline{\pbw})^{-1}\circ P_{\cU} \circ \pbw^{\nabla}:
\big(S_{\Omega_{A}}(\sections{\pi^{!}L}), \QU^{\pbw;{\nabla}}\big)\to
\big(S_{\Omega_{A}}(\Omega_{A}(B)), d^{\lightning}\big)
\end{equation*}
is a morphism of DG $\rR$-coalgebras.
By construction, the following diagram of cochain complexes commutes:
\begin{equation*}
\begin{tikzcd}
(\sections{\pi^{!}L}, \QL) \arrow{rr}{P_0} \arrow{d}{} && (\Omega_{A}(B), d^{\Bott}) \arrow{d}{} \\
\big(S_{\Omega_{A}}(\sections{\pi^{!}L}) , \QU^{\pbw;{\nabla}}\big) \arrow{rr}{P_{\cU}^{\pbw}}
&& \big(S_{\Omega_{A}}(\Omega_{A}(B)), d^{\lightning}\big)
\end{tikzcd}
\end{equation*}
where the vertical arrows are natural inclusions. 
Thus, the induced $L_{\infty}[1]$ $\rR$-morphism $P_{\cU}^{\pbw}$ has linear part $P_0$.
Since $P_{0}$ is a quasi-isomorphism, so is the induced $L_{\infty}[1]$ $\rR$-morphism $P_{\cU}^{\pbw}$.
\end{proof}

\begin{remark}\label{rem:LPC}
In fact, the $\rR$-quasi-isomorphism $P_{\cU}^{\pbw}$ in the proof of Theorem~\ref{thm:LPHA} is the projection in the graded $\cR$-linear contraction from $\big(S_{\Omega_A}(\sections{\pi^!L}),\QU^{\pbw;\nabla}\big)$ onto $\big(S_{\Omega_A}(\Omega_A(B)),d^\lightning\big)$ obtained by conjugating the contraction established in~\cite[Theorem~4.6]{MR5036937} via the respective PBW maps.
By the coalgebra perturbation lemma, its right inverse $I_{\cU}^{\pbw}$ is also a morphism of DG $\rR$-coalgebras and hence determines an $L_\infty[1]$ $\rR$-morphism between the corresponding Kapranov $L_\infty[1]$ $\rR$-algebras (see the discussion following Lemma~\ref{lem:MC}).
Moreover, the cochain map $P_{0}$ extends to a strict morphism between abelian $L_{\infty}[1]$ $\rR$-algebras 
\[ P_{0}:(\sections{\pi^{!}L}, \QL, 0, 0, \ldots) \to (\Omega_{A}(B), d^{\Bott}, 0, 0, \ldots )\, .\]

Consequently, if the Kapranov $L_{\infty}[1]$ $\rR$-algebra constructed in Theorem~\ref{thm:Liepair1} is $\rR$-linearizable, then so is the one in Theorem~\ref{thm:LAKapranov}.
Indeed, let \[F:(\sections{\pi^{!}L} , \QL, q_{2},q_{3},\ldots ) \to (\sections{\pi^{!}L} , \QL, 0,0,\ldots)\] be the linearizing isomorphism whose linear part is the identity (cf. Remark~\ref{rem:Fid}).
The composition $P_{0}\circ F\circ I^{\pbw}_{\cU}$ has the identity as its linear part, and is therefore a linearizing isomorphism for the second Kapranov $L_{\infty}[1]$ $\rR$-algebra.
\end{remark}

According to Liao \cite[Theorem~3.1]{MR4665716}, there is an isomorphism
\begin{equation*}
\Phi: H^{1}(\sections{S^2(\pi^{!}L)^{\vee}\tensor \pi^{!}L}, \liederivative{s})
\xto{\sim} H_{\CE}^{1}(A; S^{2}B^{\vee}\tensor B)
\end{equation*}
such that $\Phi \big(\alpha_{(\pi^{!}L,\liederivative{s})}\big)=\alpha_{(L,A)}$.

Combining Liao’s result with Theorems~\ref{thm:DGPBW} and~\ref{thm:LAHA} and Remark~\ref{rem:LPC}, we obtain the following corollary.

\begin{corollary}\label{cor:LPHA}
Let $(L,A)$ be a Lie pair, $B=L/A$ the quotient bundle,
and $\rR$ the DGCA $(\Omega_{A}^{\bullet},d_{\CE})$.
\begin{enumerate}
\item\label{item:LPHA1} Every Kapranov $L_{\infty}[1]$ $\rR$-algebra on $\Omega_{A}(B)$ as in Theorem~\ref{thm:LAKapranov} is linearizable if and only if its Atiyah class $\alpha_{(L,A)}$ vanishes.
\item\label{item:LPHA2} On the other hand, every induced Kapranov $L_{\infty}[1]$ $\KK$-algebra on $\Omega_{A}(B)$ is necessarily homotopy abelian.
\end{enumerate}
\end{corollary}

Consider the Lie pair
$(L,A):=(\tangentcc{X},\tangentzo{X})$
corresponding to a complex manifold $X$.
In this case, the pull-back DG Lie algebroid
$\pi^{!}L\to A[1]$
is exactly the tangent DG Lie algebroid
$\big(T(\tangentzo{X}[1]),
\liederivative{\bar{\partial}}\big)$
of the DG manifold $(\tangentzo{X}[1],\bar{\partial})$.
It is known that the Atiyah class of the Lie pair
coincides with the ordinary Atiyah class of the
holomorphic tangent bundle $T_X$ (see~\cite{MR3439229}).
Thus, we have the following:

\begin{corollary}
\label{cor:complexmanifold}
Let $X$ be a complex manifold, and let $\rR$ be the DGCA
$(\Omega_{X}^{0,\bullet},\bar{\partial})$.
\begin{enumerate}
\item \label{eq:CPX1}
The space of vector fields $\XX(\tangentzo{X}[1])$
admits a Kapranov $L_{\infty}[1]$ $\rR$-algebra structure
in which the unary bracket is $\liederivative{\bar{\partial}}$
and the binary bracket is, up to sign, a cocycle representative of the Atiyah class
of the tangent DG Lie algebroid
$\big(T(\tangentzo{X}[1]),\liederivative{\bar{\partial}}\big)$
of the DG manifold $(\tangentzo{X}[1],\bar{\partial})$.
\item \label{eq:CPX2}
The Dolbeault complex $(\Omega_{X}^{0,\bullet}(\tangentoz{X}),\bar{\partial})$
admits a Kapranov $L_{\infty}[1]$ $\rR$-algebra structure
in which the binary bracket is, up to sign, a cocycle representative of the Atiyah class
(in Dolbeault cohomology) of the holomorphic tangent bundle $T_X$.
\item
The Kapranov $L_{\infty}[1]$ $\rR$-algebras in~\ref{eq:CPX1} and~\ref{eq:CPX2} are quasi-isomorphic.
\item \label{eq:CPX4}
The Kapranov $L_{\infty}[1]$ $\rR$-algebra structure on $\Omega_{X}^{0,\bullet}(\tangentoz{X})$
is linearizable if and only if the Atiyah class of the holomorphic tangent bundle $T_X$ vanishes.
\item Nevertheless, the induced $L_{\infty}[1]$ $\CC$-algebra structure
on $\Omega_{X}^{0,\bullet}(\tangentoz{X})$ is necessarily homotopy abelian.
\end{enumerate}
\end{corollary}

In particular, when $X$ is a Kähler manifold, the Kapranov $L_{\infty}[1]$ $\rR$-algebra on $\Omega_{X}^{0,\bullet}(\tangentoz{X})$ is isomorphic to the original one introduced by Kapranov \cite[Theorem~2.6]{MR1671737}.
Thus, we have recovered the result, proved by the first author in~\cite{MR3579974,MR3622306}, that the $L_{\infty}[1]$ algebra discovered by Kapranov is homotopy abelian.

More generally, by the correspondence between Kapranov DG manifolds and Kapranov $L_\infty[1]$ $\rR$-algebras (Example~\ref{ex:KapDG}), Corollary~\ref{cor:LPHA}~\ref{item:LPHA1} provides an alternative proof of~\cite[Theorem~5.10]{MR4271478}.
Similarly, Corollary~\ref{cor:complexmanifold}~\ref{eq:CPX2} and~\ref{eq:CPX4} provide an alternative proof of~\cite[Theorem~5.25~(ii) and~(iv)]{MR4271478}.
Note that the first and third assertions in Corollary~\ref{cor:complexmanifold} were proved in~\cite{MR4393962}.
However, the last assertion in Corollary~\ref{cor:complexmanifold} is new.

\section*{Acknowledgements}
The authors thank Zhuo Chen, Kai Wang, and Maosong Xiang for helpful discussions.
They also thank the anonymous referees for their valuable suggestions, which helped improve the presentation of the paper.
Furthermore, the authors express their gratitude to the Korea Institute for Advanced Study (Xu), Pennsylvania State University (Seol), Research Institute for Mathematical Sciences (Seol and Xu), and Sapienza Università di Roma (Stiénon and Xu) for their hospitality during various stages of this work.

\section*{Declarations}
\begingroup
\renewcommand{\addcontentsline}[3]{}
\subsection*{Funding}

Seol is supported by the KIAS Individual Grants MG090801 and MG090802 at Korea Institute for Advanced Study. 
Xu received partial financial support from the National Science Foundation under award DMS-2302447 and from the Simons Foundation under award MP-TSM-00002272.

\subsection*{Competing Interests}

The authors have no relevant financial or non-financial competing interests to declare that are relevant to the content of this article.

\subsection*{Data Availability}

All data generated or analysed during this study are included in this published article.
\endgroup
\printbibliography
\end{document}